\documentclass[11pt,a4paper]{article}
\usepackage{amsmath,  amsthm, makeidx,setspace}
\usepackage{mathtools, amsfonts,amssymb}
\usepackage{fullpage}
\usepackage[colorlinks,
            linkcolor=blue,
            anchorcolor=blue,
            citecolor=blue]{hyperref}

\allowdisplaybreaks[4]

\theoremstyle{plain}
\newtheorem{theorem}{Theorem}[section]
\newtheorem{proposition}[theorem]{Proposition}
\newtheorem{lemma}[theorem]{Lemma}

\theoremstyle{definition}
\newtheorem{definition}[theorem]{Definition}

\newtheorem{remark}[theorem]{Remark}

\newtheorem{hypothesis}[theorem]{Hypothesis}

\def\H{\mathbb{H}}
\def\V{\mathbb{V}}
\def\U{\mathbb{U}}
\def\R{\mathbb{R}}
\def\E{\mathbb{E}}
\def\P{\mathbb{P}}
\def\N{\mathbb{N}}

\def\Q{\mathrm{Q}}
\def\d{\mathrm{d}}
\def\PP{\mathrm{P}}
\def\wi{\widehat}

\begin{document}

\title{ \textbf{Existence of optimal controls for stochastic partial differential equations with fully local monotone coefficients} }
\author{
  \textbf{Gaofeng Zong}\thanks{E-mail address: zonggf@sdufe.edu.cn}\\
  \small{  School of Statistics and Mathematics,  }\\
  \small{  Shandong University of Finance and Economics, Jinan, 250014, China}\\
    }

\date{}

\maketitle

\begin{abstract}
This paper deals with a stochastic optimal feedback control problem for the controlled stochastic partial differential equations. More precisely, we establish the existence of stochastic optimal feedback control for the controlled stochastic partial differential equations with fully monotone coefficients by a minimizing sequence for the control problem. Using the Faedo-Galerkin approximations, the uniform estimates and the tightness in some appropriate space for the Faedo-Galerkin approximating solution can be obtain to prove the well-posedness of the controlled stochastic partial differential equations with fully monotone coefficients. The results obtained in the present paper may be applied to various types of controlled stochastic partial differential equations, such as the controlled stochastic convection diffusion equation.
\end{abstract}
	

\medskip


\section{Introduction}\label{Introduction}\setcounter{equation}{0}

 Let  $\H$ and $\V$ be a separable Hilbert space and a reflexive Banach space,  respectively, such that the embedding $\V\subset\H$ is compact. Let $\V^*$ be the dual space of $\V$ and $\H^*$($\cong\H$) be the dual space of $\H$. The norms of $\H,\V$ and $\V^*$ are denoted by $\|\cdot\|_\H,\|\cdot\|_\V $ and $\|\cdot\|_{\V^*},$ respectively, and we have  a  Gelfand triplet
 $$	\V\subset \H\subset \V^*.$$
  Let us represent $(\cdot,\cdot)$ for the inner product in $\H$ and  $\langle \cdot,\cdot\rangle,$ the duality pairing between $\V^*$ and $\V$. Also, if $u\in\H$ and $v\in\V,$ then we have $\langle u, v\rangle=	(u, v). $ Let $\U$ be an  another separable Hilbert space, and $L_2(\U,\H)$ be the space of all Hilbert-Schmidt operators from $\U$ to $\H$ with the norm $\|\cdot\|_{L^2}$ and inner product $(\cdot,\cdot)_{L^2}$.
	
Let $(\Omega,\mathcal{F}, \P)$ be a complete probability space equipped with an increasing family of normal filtration $\{\mathcal{F}_t\}_{t\geq0}$ of $\mathcal{F}$ satisfying usual condition.
  	Let $W(\cdot)$ be a cylindrical Wiener process  on  $\U$ defined on the filtered probability space $(\Omega,\mathcal{F},\{\mathcal{F}_t\}_{t\geq 0},\P)$.

Let $T>0$ be some fixed time in this paper. Consider the following controlled stochastic partial differential equation with fully monotone coefficients in Gelfand triplet:
\begin{equation}\label{2312-22eq01.1}
  \d X_\Phi(t) = [A(t, X_\Phi(t)) + \Phi(t, X_\Phi(t))] \d t + B(t, X_\Phi(t)) \d W(t),
\end{equation}
with initial value $ X_\Phi(0)= x\in \H,$
where
\begin{eqnarray*}
  A &:&  [0,T] \times \V \to \V^*,\\
   \Phi &:&  [0,T] \times \V  \to \H,\\
   B &:&  [0,T] \times \V  \to L_2(\U,\H),
\end{eqnarray*}
 are progressively measurable, $A$ satisfies fully locally monotone condition, see condition (A2) below, and the function $\Phi$ acts as a feedback control of the equation (\ref{2312-22eq01.1}).

In this paper, we will study a stochastic optimal feedback control problem for the controlled stochastic partial differential equations \eqref{2312-22eq01.1}.
For this, let us now define the cost functional
\begin{eqnarray}\label{2312-22eq1.1.1}
  && J(\Phi):= \E\bigg[ \int_0^T \Big[ f(s, X_\Phi(s)) + g(\Phi(s, X_\Phi(s))) \Big] \d s  + h(X_\Phi(T))  \bigg],
\end{eqnarray}
where the control $\Phi$ belongs to the set of the admissible controls $\mathcal{U}$.
Our control problem is to minimise the cost functional \eqref{2312-22eq1.1.1} over $\mathcal{U}$, we denote by $(\mathcal{P})$ the problem of minimising $J$ among the admissible controls. Any $\Phi^*\in \mathcal{U}$ satisfying $$J(\Phi^*)=\inf\{J(\Phi): \Phi\in \mathcal{U}\}$$ is called an optimal feedback control.
In the cost functional (\eqref{2312-22eq1.1.1}), the weak sequentially lower semicontinuous  of the  functional is necessary for our optimal feedback  control problem $(\mathcal{P})$.

To obtain the stochastic optimal feedback control of the stochastic partial differential equations is an important question in optimal control theory and has important applications in stochastic filtering theory, in risk and game theory, in physical studies, technical, industrial and economical decisions, see Fabbri et al. (2017) for more details.
Related to the stochastic optimal control of stochastic partial differential equations, there are many investigations about the existence of optimal controls,  formulations of stochastic maximum principles, and studies about the dynamic programming principles.
Many stochastic optimal control of some special stochastic partial differential equations, such as stochastic Navier-Stokes equation, stochastic reaction-diffusion equation, stochastic semilinear equation, stochastic second grad fluids, stochastic Cahn-Hilliard equation, have been investigated in Lesei (2002), Coayla-Teran et al. (2020), Cipriano (2019), Tahraoui and Cipriano (2023), respectively.   

Our main goal in the present paper consists in the study of the optimal feedback control problem for this controlled stochastic partial differential equation with fully monotone coefficients, where fully monotone means
 $$\langle A(t,x) - A(t,y), x-y \rangle \leq [C+\rho(x)+\eta(y)]\|x-y\|_\H^2,$$
for any $x,y\in \V$, where $\rho(x), \eta(y)$ are locally bounded functions on the space $\V$.
An article closely related to the present paper is Coayla-Teran et al. (2020). The authors in Coayla-Teran et al. (2020) establish the existence of optimal controls for the controlled stochastic partial differential equation with locally monotone coefficients. If we compare the locally monotone condition used in Coayla-Teran et al. (2020), the main difference is that in the fully monotone condition both measurable functions $\rho$ and  $\eta$ can be non-zero.
In fact,  in Coayla-Teran et al. (2020), the authors have imposed a necessary condition that only one of them can be non-zero, i.e., either $\rho(x)\equiv 0$ or $\eta(y)\equiv0$, see (A2) and (C3) in Coayla-Teran et al. (2020).

It should be pointed out that all the examples considered in Coayla-Teran et al. (2020) can be covered by our framework of  fully monotone, including the controlled linear stochastic evolution equation, the controlled stochastic reaction-diffusion equation, the controlled stochastic nonlocal parabolic equation, the controlled stochastic semilinear equations.
 Moreover, our main results are also applicable to the controlled stochastic 3D tamed Navier-Stokes equations, the controlled stochastic quasilinear equation,  the controlled stochastic Cahn-Hilliard equation, the controlled stochastic Allen-Cahn-Navier-Stokes system,  which are not covered by the framework of locally monotone in Coayla-Teran et al. (2020).

The proofs of the main results are mainly inspired by the methods in Lesei (2002), Coayla-Teran et al. (2020) and  R\"{o}ckner, Shang and Zhang (2024). Here we developed these methods to the stochastic partial differential equation with fully local monotone coefficient and to study the existence of optimal feedback controls of this stochastic system. Since our considered controlled stochastic partial differential equation may be nonlinear, we are dealing with a non-convex stochastic optimization problem, which is a difficult question, see the book of Ekeland and Temam (1999) for more details about convex and non-convex optimization.
In the present paper, we follow the main idea used in Lesei (2002), in which the controlled stochastic Navier-Stokes equation was studied, to demonstrate the existence of optimal controls to the controlled stochastic partial differential equation with fully monotonicity condition.
We search for an optimal control that minimizes a given cost functional by solving an auxiliary equation of the controlled stochastic partial differential equation with fully local monotone coefficient.

The well-poseness problem of the controlled stochastic partial differential equation with fully monotonicity condition is fundamentally very difficult to solve.
Motivated by an outstanding  work of R\"{o}ckner, Shang and Zhang (2024) on stochastic partial differential equation without the control term, we consider the stochastic partial differential equation (\ref{2312-22eq01.1}) controlled by $\Phi\in \mathcal{U}$ where $\mathcal{U}:= \{\Phi: [0,T]\times \mathbb{V} \to \mathbb{H}\}$ denotes the set of the admissible controls satisfying some necessary assumption, see the following section.
The study of theory of the stochastic partial differential equation using monotone method was initially introduced by Pardoux in Pardoux (1972, 1975), and then carried out by Krylov and Rozovski\u{\i} (1979), Gy\"{o}ngy (1982), Pr\'{e}v\"{o}t and R\"{o}ckner (2007), R\"{o}ckner and Wang (2008), Zhang (2009).
A significant progress along this research direction is presented by Liu and R\"{o}ckner (2010), in which the coefficients in the stochastic partial differential equation satisfy the local monotonicity conditions.
Recently, R\"{o}ckner, Shang and Zhang (2024) established the well-posedness results for a class of stochastic partial differential equation with fully locally monotone coefficients driven by multiplicative noise, and solved an open problem for almost one decade about the well-posedness of equation with fully locally monotone coefficients driven by multiplicative noise.
For more results about stochastic partial differential equation with monotone coefficients, see Liu (2011), Liu and R\"{o}ckner (2010), Liu,  R\"{o}ckner and da Silva (2021), Xiong and Zhai (2021),  Kumar and Mohan (2022a, 2022b), Pan and Shang (2022), and therein.

The process of overcoming these difficulties caused by nonlinearity and non convexity is as follows: Firstly, we establish the well-poseness of the controlled stochastic partial differential equation with fully monotonicity condition, and we borrows the ideas from Liu and R\"{o}ckner (2015), R\"{o}ckner, Shang and Zhang (2024), Kumar and Mohan (2022a) to achieve this goal.
Secondly, we prove that the set of admissible controls is weak sequentially compact and then the weak convergence of the feedback controls imply strong convergence of the corresponding solutions.

The remainder of this paper is organized as follows. In Section 2, we state some necessary assumptions on the coefficients of  the controlled stochastic partial differential equation. Section 3 establishes the well-posedness of the controlled stochastic partial differential equation with fully monotonicity condition, in this section, the uniform estimates and the tightness in some appropriate space  for the Faedo-Galerkin approximating solution are also established to achieve our goal.
In Section 4, we present some convergence results to prove the existence of optimal feedback controls.

\section{Preliminaries}\setcounter{equation}{0}

We start with  some basic definitions related to the operators in the controlled stochastic partial differential equation  (\ref{2312-22eq01.1})  with fully monotone coefficients.
\begin{definition}\label{2312-22def1}
	An operator $A$ from $\V$ to $\V^*$ is said to be  pseudo-monotone if the following condition holds:  for any sequence $\{u_n\}$ with the weak limit $u$ in $\V$ and
	\begin{align}\label{2312-22eq02.1}
		\liminf_{n\to\infty}		\langle A(u_n),u_n-u\rangle\geq 0,
	\end{align}imply that
	\begin{align}\label{3.3}
		\limsup_{n\to\infty}	\langle A(u_n),u_n-v\rangle \leq \langle A(u),u-v \rangle, \ \text{ for all } \ v\in\V.
	\end{align}
\end{definition}

We consider the following assumptions on the coefficients $A$, $\Phi$, and $B$  in the controlled stochastic partial differential equation  (\ref{2312-22eq01.1}).
Let $f_A\in L^1(0,T;\mathbb{R}_+)$.  For the  operator $A$, we will assume:
\begin{hypothesis}\label{2312-22hypo1}
Let  $\beta\in (1,\infty)$.
	\begin{itemize}
	\item[(A.1)] (Hemi-continuity). The map $\mathbb{R}\ni\lambda \mapsto \langle A(t,u+\lambda v),w\rangle \in\R$ is continuous for any $u,v,w\in \V$ and for a.e. $t\in[0,T]$.
	\item[(A.2)] (Local monotonicity).  There exist  non-negative constants $\zeta$ and $C$ such that for any $u,v\in\V$ and a.e. $t\in[0,T]$,
	\begin{equation}\label{2312-22eq3.5}
2\langle A(t,u)-A(t,v),u-v\rangle    \leq \big[f_A(t)+\rho(u)+\eta(v)\big]\|u-v\|_{\H}^2,
	\end{equation}
   where $\rho$ and $\eta$ are two measurable functions from $\V$ to $\R$ satisfy $|\rho(u)|+|\eta(u)| \leq C(1+\|u\|_{V}^\beta)(1+\|u\|_{\H}^\zeta)$.
\item[(A.2)$'$] (General local monotonicity). For any constant $r>0$, there exists a function $M_{\cdot}(r)\in L^1(0,T;\R_+)$  such that for any $\|u\|_\V\vee \| v\|_\V\leq r$ and a.e. $t\in[0,T]$,
\begin{align}\label{2312-22eq3.6}
	\langle A(t,u)-A(t,v),u-v\rangle \leq M_t(r)\|u-v\|_{\H}^2.
\end{align}	
	\item[(A.3)]	(Coercivity). There exists a positive constant $C$ such that for any $u\in\V$ and a.e.  $t\in[0,T]$,
	\begin{align}\label{2312-22eq3.7}
		2\langle A(t,u),u\rangle
		\leq f_A(t)(1+\|u\|_{\H}^2)-C\|u\|_{\V}^\beta.
	\end{align}
\item[(A.4)] (Growth). There exist non-negative constants  $\alpha$ and $C$ such that for any $u\in\V$ and a.e. $t\in[0,T]$,
\begin{align}\label{2312-22eq3.8}
	\|A(t,u)\|_{\V^*}^{\frac{\beta}{\beta-1}}\leq (f_A(t)+C\|u\|_{\V}^\beta)(1+\|u\|_{\H}^\alpha).
\end{align}
	\end{itemize}
\end{hypothesis}

\begin{remark}
  In R\"{o}ckner, Shang and Zhang (2024), the authors introduced a general local monotonicity condition, which is weaker than (A.2) in Hypothesis  \ref{2312-22hypo1}, and then established the existence of the probabilistic strong solution for the stochastic partial differential equations without control terms. In present paper, we don't intend to use this weaker condition as our main focus is on analyzing control problems.
\end{remark}

For the  operator  $B$, we assume as follows.
\begin{hypothesis}\label{2312-22hypo2}
Let $g_B\in L^1(0,T;\R_+)$.
\begin{itemize}
\item[(B.1)] Let $f_A, \rho, \eta$ given as in (A.2),
  $$ \|B(t,u)-B(t,v)\|_{L_2}^2
		 \leq \big[f_A(t)+\rho(u)+\eta(v)\big]\|u-v\|_{\H}^2.$$
\item[(B.2)] For any sequence $\{u_m\}_{m=1}^\infty$ and $u$ in $\V$ with $\|u_m-u\|_{\H}\to 0$ as $m\to \infty$, we have
 \begin{align}\label{2312-22eq3.9}
	\|B(t,u_m)-B(t,u)\|_{L_2} \to 0,\ \text{ for a.e. } \ t\in[0,T].
\end{align}
\item[(B.3)]For any $p\geq 2$, any $u\in \V$ and a.e. $t\in[0,T]$,
\begin{align}\label{2312-22eq3.10}
	\|B(t,u)\|_{L_2}^p  \leq g_B(t)(1+\|u\|_{\H}^p).
\end{align}
	\end{itemize}
\end{hypothesis}

\begin{remark}
  In R\"{o}ckner, Shang and Zhang (2024), the authors established the uniqueness of the probabilistic strong solution for the stochastic partial differential equations without control terms under the condition (B.2) in Hypothesis  \ref{2312-22hypo2}.
However, condition (B.1) is also essential for our control problems, so for simplicity, in present paper we use condition (B.1).
The growth condition  (B.3) in Hypothesis  \ref{2312-22hypo2} is stronger than the growth condition  (H5) in R\"{o}ckner, Shang and Zhang (2024), since this growth condition  (B.3) is required in our analysis of auxiliary control processes, however, we believe that this is not an essential change.
\end{remark}

We will assume that the control coefficient $\Phi$ in the set of the admissible controls $\mathcal{U}$ for the controlled stochastic partial differential equation satisfies the following conditions: 
\begin{hypothesis}\label{2312-22hypo3}
Let $f_\Phi\in L^1(0,T; \mathbb{R}^+)$,
  \begin{itemize}
\item[(C.1)] (Lipschitz continuity). for some positive constant $\eta$, $\|\Phi(0,0)\|^2\leq \eta,$ and  for all $s,t\in [0,T]$, $x,y\in \mathbb{H}$,
    \begin{equation}\label{2312-25eq2.2}
      \|\Phi(t,x)- \Phi(s,y)\|^2\leq \lambda |t-s|^2 + \alpha\|x-y\|^2, 
    \end{equation}
    where $\lambda,\alpha$ are positive constants.
   \item[(C.2)] (Local monotonicity). There exist non-negative constants $\zeta$ and $C$ such that for any $u,v\in \mathbb{H}$ and a.e. $t\in[0,T]$,
    \begin{equation}\label{2312-25eq1.2.3}
      2( \Phi(t, u)- \Phi(t,v), u-v ) \leq  f_\Phi(t)  \|u-v\|_{\mathbb{H}}^2,
    \end{equation}
    \item[(C.3)] (Coercivity).  There exists a positive constant $C$ such that for any $u\in \mathbb{H}$ and a.e. $t\in[0,T]$,
    \begin{equation}\label{2312-25eq1.2.4}
      2( \Phi(t, u), u ) \leq f_\Phi(t)( 1+ \|u\|_{\mathbb{H}}^2).
    \end{equation}
    \item[(C.4)] (Growth).  There exists non-negative constants $\alpha$ and $C$ such that for any $p\geq 2$,  $u\in \mathbb{H}$ and a.e. $t\in[0,T]$,
    \begin{equation}\label{2312-25eq1.2.5}
      \| \Phi(t, u)\|_{\mathbb{H}}^p \leq C(f_\Phi(t) +\|u\|_{\mathbb{H}}^p).
    \end{equation}
  \end{itemize}
\end{hypothesis}

Furthermore, we will assume that the coefficients in the cost functional satisfy the following conditions.
Whenever the integral in the cost functional \eqref{2312-22eq1.1.1} exist and are finite, we will assume that the maps involved in \eqref{2312-22eq1.1.1} satisfy:
\begin{hypothesis}\label{2312-22hypo14}
\begin{itemize}
\item[(F.1)]   $f: [0,T]\times \V\to \R_+$, $g:\H\to \R_+$, and $h:\V\to \R_+$.
\item[(F.2)]  The mappings $g$ and $u\in L^2(0,T; \H) \mapsto \int_0^T g(u(s)) \d s$ are weak sequentially lower semi-continuous.
\item[(F.3)]  The mappings $f, h$ and $X\in L^2(0,T; \V) \mapsto \int_0^T f(s,X(s)) \d s$ are  sequentially lower semi-continuous.
	\end{itemize}
\end{hypothesis}

\begin{remark}\label{2312-22rem4}
The sequentially lower semi-continuous for $f$ implies that
$$X_n\to X ~ \text{ strongly in $\V$ }  \Rightarrow  f(X)\leq \lim_{n\to\infty}\inf f(X_n). $$
However,	the weak sequentially lower semi-continuous for $g$ implies that
$$u_n\rightharpoonup u ~ \text{ weakly in $\H$ }  \Rightarrow  g(u)\leq \lim_{n\to\infty}\inf g(u_n). $$
\end{remark}

\section{Well-posedness results}\setcounter{equation}{0}

Let us now provide the definition of probabilistically strong solution for the controlled stochastic partial differential equation (\ref{2312-22eq01.1}).

\begin{definition}\label{2312-25def2.1}
  A   \emph{probabilistically weak solution} to the equation  (\ref{2312-22eq01.1})  is a system $$((\wi{\Omega},\wi{\mathcal{F}},\{\wi{\mathcal{F}}_t\}_{t\geq0},\wi{\P}),\wi{X},\wi{W}),$$ where
	\begin{itemize}
		\item[(1)] $(\wi{\Omega},\wi{\mathcal{F}},\{\wi{\mathcal{F}}_t\}_{t\geq0},\wi{\P})$ is a filtered probability space with the filteration $\{\wi{\mathcal{F}}_t\}_{t\geq0}$,
		\item[(2)] $\wi{W}$ is a $\U$-cylindrical Wiener process on $(\wi{\Omega},\wi{\mathcal{F}},\{\wi{\mathcal{F}}_t\}_{t\geq0},\wi{\P})$,
		\item[(3)] $\wi{X}:[0,T]\times \wi{\Omega} \to \H$ is a predictable process with $\wi{\P}$-a.s., paths
		\begin{align*}
			\wi{X}(\cdot,\omega)\in \mathrm{C}([0,T];\H)\cap L^\beta(0,T;\V),
		\end{align*}such that for all $t\in[0,T]$ and  for all $v\in\V$, the following holds:
		\begin{align}\label{3.12}\nonumber
		(\wi{X}(t),v)&=(x,v) +\int_0^t\langle A(s,\wi{X}(s)),v\rangle \d s   +\int_0^t  ( \Phi(s,\wi{X}(s)),v) \d s  \\
                      & \ + \int_0^t(B(s,\wi{X}(s))\d\wi{W}(s),v), \     \wi{\P}\text{-a.s.}
		\end{align}
	\end{itemize}
\end{definition}

Let us provide the definition of pathwise strong probabilistic (analytically weak in PDE theory) solution for the controlled stochastic partial differential equation (\ref{2312-22eq01.1}).
\begin{definition}
	We are given a stochastic basis  $((\Omega,\mathcal{F},\{\mathcal{F}_t\}_{t\geq0}, \P),X, W)$ and  $x\in\H$.
Then, equation (\ref{2312-22eq01.1}) has a   \emph{pathwise probabilistically strong solution}  if and only if there exists a progressively measurable process $X_\Phi:[0,T]\times \Omega\to \H$ with $\P$-a.s., paths
 \begin{align*}
		X_\Phi(\cdot,\omega) \in \mathrm{C}([0,T]; \H)\cap L^\beta(0,T;\V),
	\end{align*}
and the following equation
\begin{align*}
		(X_\Phi(t), v)&=(x, v) +\int_0^t\langle A(s,X_\Phi(s)), v\rangle \d s  +\int_0^t ( \Phi(s,X_\Phi(s)), v) \d s \\
                 &   + \int_0^t(B(s, X_\Phi(s))\d W(s), v), \; \text{ for all } v\in\V,
\end{align*}
holds $\P$-a.s., for all $t\in[0,T]$.
\end{definition}

\begin{definition}
\begin{description}
  \item[(1)] For $i=1,2$, let us consider $X_i$ be any solution on the stochastic basis $$((\Omega,\mathcal{F},\mathcal{F}_{t\geq0}, \P), X_i, W)$$  to the equation (\ref{2312-22eq01.1}) with $X_i(0)=x$. Then, the solutions of the equation (\ref{2312-22eq01.1}) are pathwise unique if and only if
	\begin{align*}
		\P\big\{X_1(t)= X_2(t),\text{ for all } t\geq 0\big\}=1.
	\end{align*}
  \item[(2)] We say that solutions of the equation (\ref{2312-22eq01.1}) are unique in law if and only if the following holds:
	 If $((\Omega_i,\mathcal{F}_t,\{\mathcal{F}_i\}_{t\geq0}, \P_i), X_i, W_i)$ for $i=1,2$ are two  solutions to the equation (\ref{2312-22eq01.1}) with $X_i(0)=x$, for $i=1,2$, then $\mathcal{L}_{P_1}(X_1)=\mathcal{L}_{P_2}(X_2)$.
\end{description}
\end{definition}

Here, Pathwise uniqueness means indistinguishability. The well-known Yamada-Watanabe theorem  states that weak existence and pathwise uniqueness imply strong existence and weak uniqueness.

In R\"{o}ckner, Shang, and Zhang (2024), the results on the existence of probabilistically weak and strong solution of the equation (\ref{2312-22eq01.1}) without control hold.

\begin{theorem}\label{2312-22TH2.11}
  Assume that the embedding $\V\subset \H$ is compact and Hypothesis \ref{2312-22hypo1} (A.1)-(A.4), Hypothesis \ref{2312-22hypo2} (B.1)-(B.3) and Hypothesis \ref{2312-22hypo3} (C.1)-(C.4) hold. Then, for any initial data $x\in\H$, there exists a probabilistic weak solution $X_\Phi$ to the equation (\ref{2312-22eq01.1}). furthermore, for any $p\geq 2$, the following estimate holds:
  \begin{equation}\label{2312-25eq2.13}
    \E \Big[  \sup_{t\in [0,T]} \|X_\Phi(t)\|_\H^p   + \Big( \int_0^t \|X_\Phi(t)\|_\V^\beta \d t\Big)^{\frac{p}{2}} \Big]  < C(1+\|x\|_{\H}^p).
  \end{equation}
  Moreover, if Hypothesis 2.4 (H.2) holds, then solution of the  equation (\ref{2312-22eq01.1}) is pathwise unique and hence there exists a unique probabilistic strong solution to the  equation (\ref{2312-22eq01.1}) with the initial data $x\in \H$.
\end{theorem}

In order to obtain the existence of a probabilistically weak solution, we start with the construction of an approximating solutions using a Faedo-Galerkin approximation and then we prove the tightness of laws of the approximating solutions in the approximating space.

Let $\{e_j\}_{j=1}^\infty\subset \V$ be an orthonormal  basis of $\H$. Let  $\H_m$  be a finite dimensional space  spanned by $\{e_1,e_2,\ldots,e_m\}$.  We define a projection $\PP_m:\V^*\to \H_m$  by
\begin{align}\label{2312-25eq3.15}
	\PP_m u :=\sum_{j=1}^{m}\langle u,e_j\rangle e_j.
\end{align}
 Clearly, the restriction of this projection denoted by $\PP_m\big|_\H$ is just the orthogonal projection of $\H$ onto $\H_m$. Since $\{\phi_j\}_{j=1}^{\infty}$ is an orthonormal basis of the Hilbert space $\U$, let us set
\begin{align}\label{2312-25eq3.16}
	W_m(t)=\Q_m W(t)=\sum_{j=1}^{m}\langle W(t),\phi_j\rangle \phi_j,
\end{align}
where $\Q_m$ is the orthogonal projection onto $\text{span}\{\phi_1,\phi_2,\ldots,\phi_m\}$ in $\U$.

Now, we consider  the following stochastic differential equation in the finite-dimensional  space $\H_m$, for any $m\geq 1$
\begin{eqnarray}\label{2312-25eq3.17}
	Y_m(t)
 &=& \PP_m x+\int_{0}^{t}\PP_m A(s, Y_m(s))\d s  +\int_{0}^{t}\PP_m \Phi(s, Y_m(s))\Q_m\d s     \nonumber \\
 && \quad  +\int_{0}^{t}\PP_m B(s, Y_m(s))\Q_m\d W(s) .
\end{eqnarray}
The existence and uniqueness of the strong solution of finite dimensional system has been discussed in Theorem 1, Gy\"ongy (1982).
 We have the following uniform energy estimate for $\{Y_m\}$.

\begin{lemma}\label{2312-25lem1}
	For any $p\geq 2$, there exists a constant $C_p>0$ such that
\begin{align}\label{2312-25eq3.18}		
		\sup_{m\in\mathbb{N}}\left\{\E\bigg[\sup_{0\leq t\leq T}\|Y_m(t)\|_{\H}^p\bigg]+\E\bigg[\bigg(\int_{0}^{T}\|Y_m(t)\|_{\V}^\beta\d t\bigg)^{\frac{p}{2}}\bigg]\right\}\leq  C_p(1+\|x\|_{\H}^p).
\end{align}
\end{lemma}
\begin{proof}
%
	Applying It\^o's formula to the real-valued process $\|Y_m(\cdot)\|_{\H}^p$, we find
\begin{align}\label{2312-25eq3.20}
\|Y_m(t)\|_{\H}^p &= \|\PP_m x\|_{\H}^p + p \int_{0}^{t}\|Y_m(s)\|_{\H}^{p-2} \langle A(s,Y_m(s)),Y_m(s)\rangle  \d s\nonumber\\
&\quad +\frac{p}{2}\int_{0}^{t}\|Y_m(s)\|_{\H}^{p-2}\big[2( \Phi(s,Y_m(s)),Y_m(s))
           +\|\PP_m B(s, Y_m(s))\Q_m\|_{L_2}^2\big]\d s \nonumber\\
& \quad +\frac{p(p-2)}{2}\int_{0}^{t}\|Y_m(s)\|_{\H}^{p-4}\|Y_m(s)\circ \PP_m B(s, Y_m(s))\Q_m\|_{\U}^2\d s  \nonumber\\
&\quad +p\int_{0}^{t}\|Y_m(s)\|_{\H}^{p-2}(  B(s, Y_m(s))\Q_m\d W(s), Y_m(s)).
\end{align}
Using (A.3) in Hypothesis \ref{2312-22hypo1}, (B.3)  in Hypothesis \ref{2312-22hypo2}, and (C.3)  in Hypothesis \ref{2312-22hypo3}, for \eqref{2312-25eq3.20}, we obtain
\begin{align}\label{2312-25eq3.020}
&\|Y_m(t)\|_\H^p  +\frac{pC}{2}\int_0^t\|Y_m(s)\|_\V^\beta\|Y_m(s)\|_\H^{p-2}\d s \nonumber\\
\leq &\|\PP_m x\|_\H^p +  C_p  \int_{0}^{t}[f_A(s) +f_\Phi(s) +g_B(s)](1+ \|Y_m(s)\|_\H^2)\|Y_m(s)\|_\H^{p-2} \d s    \nonumber \\
& 	+ p \int_0^t  \|Y_m(s)\|_\H^{p-2}\big( B(s,Y_m(s))\Q_m\d W(s),Y_m(s)\big).
\end{align}

Due to
\begin{eqnarray}\label{2312-25eq3.1020}
    &&   \int_{0}^{t}[f_A(s) +f_\Phi(s) +g_B(s)](1+ \|Y_m(s)\|_\H^2)\|Y_m(s)\|_\H^{p-2} \d s  \nonumber \\
   &\leq&     \int_{0}^{t}\big([f_A(s) +f_\Phi(s) +g_B(s)]  \|Y_m(s)\|_\H^{p}  + [f_A(s) +f_\Phi(s) +g_B(s)](1+ \|Y_m(s)\|_\H^{p})  \big)\d s   \nonumber \\
   &\leq&     \int_{0}^{t} [f_A(s) +f_\Phi(s) +g_B(s)] \d s  +  2 \int_{0}^{t} [f_A(s) +f_\Phi(s) +g_B(s)] \|Y_m(s)\|_\H^{p} \d s,
\end{eqnarray}
substituting \eqref{2312-25eq3.1020} in \eqref{2312-25eq3.020}, we obtain
\begin{eqnarray}\label{2312-25eq3.2020}
&&\|Y_m(t)\|_\H^p  +\frac{pC}{2}\int_0^t\|Y_m(s)\|_\V^\beta\|Y_m(s)\|_\H^{p-2}\d s \nonumber\\
&\leq &\|\PP_m x\|_\H^p   + C \int_{0}^{t} F(s) \d s  + C\int_{0}^{t} F(s) \|Y_m(s)\|_\H^{p} \d s   \nonumber \\
&& 	+ p \int_0^t  \|Y_m(s)\|_\H^{p-2}\big(  B(s,Y_m(s))\Q_m\d W(s),Y_m(s)\big),
\end{eqnarray}
where
\begin{equation}\label{2312-25eq3.3020}
F(s) := f_A(s) +f_\Phi(s) +g_B(s).
\end{equation}

Define a sequence of stopping times as follows
$$\tau_N^m := T\wedge\inf\{  t\geq 0: \|Y_m(t)\|_H  >N\}.$$
Then $\tau_N^m\to T,$ $\P$-a.s. as $N\to \infty$ for every $m$.
Taking the supremum over $t\leq r\wedge \tau_N^m$ where for any $r\in[0,T]$ and then taking expectations on both sides of the above inequality \eqref{2312-25eq3.2020} yield
\begin{align}\label{2312-25eq3.22}
  &\E\bigg[\sup_{t\in[0,r\wedge \tau_N^m]}\|Y_m(t)\|_\H^p\bigg]+\frac{pC}{2}\E\bigg[\int_{0}^{r\wedge \tau_N^m}\|Y_m(s)\|_{\V}^\beta\|Y_m(s)\|_\H^{p-2}\d s\bigg]  \nonumber \\
\leq & \|x\|_\H^p+C\int_{0}^{t}F(s)\d s+C\E\bigg[\int_{0}^{r\wedge \tau_N^m}F(s)\|Y_m(s)\|_\H^p\d s\bigg]  \nonumber \\
     &   +p\E\bigg[\sup_{t\in[0,r\wedge \tau_N^m]}\bigg|\int_{0}^{t}\|Y_m(s)\|_{\H}^{p-2}\big(B(s, Y_m(s))\Q_m\d W(s),Y_m(s)\big)\bigg|\bigg].
\end{align}
Now, we consider the last term in \eqref{2312-25eq3.22} and estimate it using Hypothesis \ref{2312-22hypo2} (B.3), Burkholder-Davis-Gundy inequality, H\"older's and Young's inequalities as
\begin{align}\label{2312-25eq3.23}
& p\E\bigg[\sup_{t\in[0,r\wedge \tau_N^m]}\bigg|\int_{0}^{t}\|Y_m(s)\|_{\H}^{p-2}\big(B(s, Y_m(s))\Q_m\d W(s),Y_m(s)\big)\bigg|\bigg]    \nonumber \\	
& \leq  C_p\E\bigg[\bigg(\int_{0}^{r\wedge \tau_N^m}\|Y_m(s)\|_{\H}^{2p-2}\|B(s,Y_m(s))\|_{L_2}^2\d s\bigg)^{\frac{1}{2}}\bigg]   \nonumber \\
& \leq 	C_p\E\bigg[\bigg(\sup_{s\in[0,r\wedge \tau_N^m]}\|Y_m(s)\|_\H^p \int_{0}^{r\wedge \tau_N^m}\|Y_m(s)\|_{\H}^{p-2}\|B(s, Y_m(s))\|_{L_2}^2\d s\bigg)^{\frac{1}{2}}\bigg]   \nonumber \\
& \leq 	\epsilon\E\bigg[\sup_{s\in[0,r\wedge \tau_N^m]}\|Y_m(s)\|_\H^p\bigg]  + C_{\epsilon, p}\E\bigg[\int_{0}^{r\wedge \tau_N^m}\|Y_m(s)\|_{\H}^{p-2}\|B(s,Y_m(s))\|_{L_2}^2\d s\bigg]   \nonumber \\
& \leq \epsilon\E\bigg[\sup_{s\in[0,r\wedge \tau_N^m]}\|Y_m(s)\|_\H^p\bigg]  +  C_{\epsilon,p}   \int_{0}^{T}g(s)\d s+C_{\epsilon,p}\E\bigg[\int_{0}^{r\wedge \tau_N^m}g(s)\|Y_m(s)\|_{\H}^p\d s\bigg],
\end{align}
where argument similar to \eqref{2312-25eq3.1020} is applied in the last inequality and $\epsilon>0$. Combining  \eqref{2312-25eq3.22} and  \eqref{2312-25eq3.23}, and choosing an appropriate parameter $\epsilon$ and applying Gronwall's inequality, we obtain
\begin{align}\label{2312-25eq3.24}
&\E\bigg[\sup_{s\in[0,r\wedge \tau_N^m]}\|Y_m(s)\|_\H^p\bigg]   +  C  \E\bigg[ \int_0^{r\wedge \tau_N^m}\|Y_m(s)\|_\V^\beta\|Y_m(s)\|_\H^{p-2}\d s \bigg] \nonumber\\
\leq & C\|x\|_\H^p  + C\int_{0}^{t}F(s)\d s+C\E\bigg[\int_{0}^{r\wedge \tau_N^m}F(s)\|Y_m(s)\|_\H^p\d s\bigg]  \nonumber \\
&     +  C_{\epsilon,p}   \int_{0}^{T}g(s)\d s+C_{\epsilon,p}\E\bigg[\int_{0}^{r\wedge \tau_N^m}g(s)\|Y_m(s)\|_{\H}^p\d s\bigg] \nonumber \\
\leq & C\|x\|_\H^p  + C_{\epsilon,p}\int_{0}^{t}F(s)\d s+C_{\epsilon,p}\E\bigg[\int_{0}^{r\wedge \tau_N^m}F(s)\|Y_m(s)\|_\H^p\d s\bigg]  \nonumber \\
\leq & C\bigg(\|x\|_\H^p  + \int_{0}^{T}F(s)\d s\bigg) \exp\bigg(  C \int_{0}^{T}F(s)\d s \bigg).
\end{align}
Letting $N\to\infty$, and applying Fatou's lemma, we find, for all $p\geq 2$,
\begin{eqnarray}\label{2312-25eq3.25}
   && \sup_{m\in \mathbb{N}} \Bigg\{\E\bigg[\sup_{t\in[0,T]}\|Y_m(t)\|_\H^p\bigg]   +    \E\bigg[ \int_0^{T}\|Y_m(s)\|_\V^\beta\|Y_m(s)\|_\H^{p-2}\d s \bigg] \Bigg\}   \nonumber\\
&\leq &   C (1+ \|x\|_\H^p),
\end{eqnarray}
where, we used $f_A, g_B$ and $f_\Phi$ belong to $L^1(0,T;\R^+)$.

\

Again, we apply It\^{o}'s formula to the process $\|Y_m(t)\|_\H^2$, to find
\begin{eqnarray}\label{2312-25eq3.26}
 && \|Y_m(t)\|_{\H}^2\nonumber \\
 &= &\|\PP_m x\|_{\H}^2 + \int_{0}^{t}\big[2\langle A(s,Y_m(s)),Y_m(s)\rangle  +  2(\Phi(s,Y_m(s)), Y_m(s)) \big]\d s   \nonumber \\
 &&  +  \int_{0}^{t}\|\PP_m B(s, Y_m(s))\Q_m\|_{L_2}^2 \d s  + 2\int_{0}^{t}(B(s, Y_m(s))\Q_m\d W(s),Y_m(s))
\end{eqnarray}
 Using Hypothesis \ref{2312-22hypo1} (A.3),  Hypothesis \ref{2312-22hypo2} (B.3), and  Hypothesis \ref{2312-22hypo3} (C.3) in \eqref{2312-25eq3.26},  it follows that
\begin{eqnarray}\label{2312-25eq3.27}
  &&\|Y_m(t)\|_\H^2+C\int_{0}^{t}\|Y_m(s)\|_\V^\beta\d s  \nonumber\\
  & \leq&  \|\PP_m x\|_\H^2+\int_{0}^{t}\left\{F(s)\|Y_m(s)\|_\H^2+F(s)\right\}\d s  \nonumber\\
  &&\quad   +2\int_{0}^{t}(B(s, Y_m(s))\Q_m\d W(s), Y_m(s)),
\end{eqnarray}
where $F(s)$ defined by \eqref{2312-25eq3.3020}.
Using the same argument above for \eqref{2312-25eq3.27}, by virtue of \eqref{2312-25eq3.25},  it is easy to obtain
\begin{eqnarray}\label{2312-25eq3.28}
    \sup_{m\in \mathbb{N}}     \E\bigg[ \int_0^{T}\|Y_m(s)\|_\V^\beta \d s \bigg]^{\frac{p}{2}}   \leq   C (1+ \|x\|_\H^p).
\end{eqnarray}
Combining  \eqref{2312-25eq3.25} and \eqref{2312-25eq3.28}, we get our desired result \eqref{2312-25eq3.18}, which completes the proof.
\end{proof}

Next, we establish the tightness property of the laws of the sequence $\{Y_m\}$.
For this, define stopping times as follows,
\begin{eqnarray}\label{2312-25eq3.29}
  \tau_m^N  :=T  \wedge\inf\{t\ge0:\|Y_m(t)\|_\H^2>N\}  \wedge\inf\bigg\{t\geq0: \int_0^t\|Y_m(s)\|_\V^\beta\d s >N\bigg\},
\end{eqnarray}
 with the convection that infimum of a void set is infinite.

\begin{lemma}
  For the stopping times $\tau_m^N$, we have
 \begin{equation}\label{2312-25eq3.30}
  \lim_{N\to\infty}\sup_{m\in \mathbb{N}}\P(\tau_m^N < T)=0.
 \end{equation}
\end{lemma}
\begin{proof}
  This inequality \eqref{2312-25eq3.30} can be easily obtained  by Markov's inequality and Lemma \ref{2312-25lem1}.
\end{proof}

Let now prove the tightness property of the laws of $\{Y_m\}$.
\begin{lemma}\label{2312-22lem2}
	The set of measures $\{\mathcal{L}(Y_m):m\in\mathbb{N}\}$ is tight on $L^\beta(0,T;\H)$ for $\beta\in(1,\infty)$.
\end{lemma}
\begin{proof}
Due to \eqref{2312-25eq3.18}, we have
	\begin{align}\label{2312-25eq3.34}
		\sup_{m\in\mathbb{N}}\E\bigg[\int_{0}^{T}\|Y_m(t)\|_\V^\beta\d t\bigg]<\infty,
	\end{align}
using   Lemma 5.2 in R\"{o}ckner, Shang, Zhang (2024), it is enough to prove for any $\epsilon>0$
\begin{align}\label{2312-25eq3.35}
	\lim_{\delta\to 0^+}\sup_{m\in\mathbb{N}}  \P\bigg(\int_0^{T-\delta}\|Y_m(t+\delta)- Y_m(t)\|_\H^\beta\d t>\epsilon\bigg)=0.
\end{align}
Fix $Y_m^N(t):=Y_m(t\wedge\tau_m^N)$.
An application of Markov's inequality yields
\begin{eqnarray}\label{2312-25eq3.36}
    &&   \P\bigg(\int_0^{T-\delta}\|Y_m(t+\delta)-Y_m(t)\|_\H^\beta\d t>\epsilon\bigg)  \nonumber \\
  &\leq& \P  \bigg(\int_0^{T-\delta}\|Y_m(t+\delta)- Y_m(t)\|_\H^\beta\d t>\epsilon,   \tau_m^N=T  \bigg)     +  \P(\tau_m^N <T)   \nonumber \\
  &\leq& \frac{1}{\epsilon}\E\bigg[\int_0^{T-\delta}\|Y_m^N(t+\delta)-Y_m^N(t)\|_\H^\beta\d t\bigg]   + \P(\tau_m^N<T).
\end{eqnarray}
If we manage to show that for any fixed $N>0$
\begin{align}\label{2312-25eq3.37}
	\lim_{\delta\to0+}\sup_{m\in\mathbb{N}}\E\bigg[\int_0^{T-\delta}\|Y_m^N(t+\delta)-Y_m^N(t)\|_\H^\beta\d t\bigg]=0,
\end{align}
then, in light of \eqref{2312-25eq3.30}, letting $\delta\to0$ and then $N\to\infty$ in \eqref{2312-25eq3.36}, we get \eqref{2312-25eq3.35}, which completes the proof of tightness of laws of $\{Y_m\}$ in $L^\beta(0,T;\H)$.
Thus, if we establish \eqref{2312-25eq3.37}, then we are done.
To complete this, we divide the proof in two parts which depends on the values of $\beta$, that is, $1<\beta\leq 2$ and $\beta>2$.

\

\noindent \textbf{The case of $\beta\in(1,2]$.}
 Applying It\^o's formula for $\|Y_m^N(t+\delta)-Y_m^N(t)\|_\H^2$, we get
 \begin{eqnarray}\label{2312-25eq3.38}
&&   \E\bigg[\|Y_m^N(t+\delta)-Y_m^N(t)\|_\H^2\bigg]  \nonumber   \\
&=&  \E\bigg[\int_{t\wedge\tau_m^N}^{(t+\delta)\wedge\tau_m^N}  2\langle A(s,Y_m(s)),Y_m(s)-Y_m(t\wedge\tau_m^N)\rangle \d s\bigg]  \nonumber   \\
&& + \E\bigg[\int_{t\wedge\tau_m^N}^{(t+\delta)\wedge\tau_m^N}  2\Big( \Phi(s,Y_m(s)),Y_m(s)-Y_m(t\wedge\tau_m^N)\Big) \d s\bigg]  \nonumber   \\
&& + \E\bigg[\int_{t\wedge\tau_m^N}^{(t+\delta)\wedge\tau_m^N}   \|\PP_m B(s,Y_m(s))\Q_m\|_{L_2}^2\d s\bigg].
 \end{eqnarray}
Integrating \eqref{2312-25eq3.38} with respect to $t$,  and applying Fubini's theorem we conclude that
\begin{eqnarray}\label{2212-25eq3.39}
&&  \E\bigg[\int_0^{T-\delta}\|Y_m^N(t+\delta)-Y_m^N(t)\|_\H^2\d t\bigg]   \nonumber   \\
&=&  \E\bigg[\int_0^{T-\delta}\bigg\{ \int_{t\wedge\tau_m^N}^{(t+\delta)\wedge\tau_m^N} \bigg(2\langle A(s,Y_m(s)),Y_m(s)\rangle   \nonumber   \\
 && ~~~~~~~~~~ + 2( \Phi(s,Y_m(s)),Y_m(s))  +  \|\PP_m B(s,Y_m(s))\Q_m\|_{L_2}^2   \bigg)\d s\bigg\}\d t\bigg] \nonumber   \\
&&  - 2\E\bigg[\int_0^{T-\delta}\bigg\{\int_{t\wedge\tau_m^N}^{(t+\delta)\wedge\tau_m^N} \Big[ \langle A(s,Y_m(s)),Y_m(t\wedge\tau_m^N)\rangle   \nonumber   \\
 && ~~~~~~~~~~  + ( \Phi(s,Y_m(s)),Y_m(t\wedge\tau_m^N))\Big] \d s\bigg\}\d t\bigg]  \nonumber   \\
&=& \E\bigg[\int_0^{T\wedge\tau_m^N} \bigg(\int_{0\vee(s-\delta)}^s \chi_{\{\tau_m^N >t\}} \d t\bigg) \bigg(2\langle A(s,Y_m(s)),Y_m(s)\rangle  \nonumber   \\
 && ~~~~~~~~~~  + 2( \Phi(s,Y_m(s)),Y_m(s)) +\|\PP_m B(s, Y_m(s))\Q_m\|_{L_2}^2 \bigg) \d s  \nonumber   \\
 && - 2\E\bigg[ \int_0^{T\wedge\tau_m^N}\bigg(\int_{0\vee(s-\delta)}^s  \chi_{\{\tau_m^N>t\}} \langle A(s,Y_m(s)),Y_m(t\wedge\tau_m^N)\rangle    \nonumber   \\
 && ~~~~~~~~~~+ \int_{0\vee(s-\delta)}^s  \chi_{\{\tau_m^N>t\}} (\Phi(s,Y_m(s)),Y_m(t\wedge\tau_m^N))\bigg) \d t\d s \bigg]   \nonumber   \\
&=:&  I_1+I_2.
\end{eqnarray}
Using Hypothesis \ref{2312-22hypo1} (A.3),  Hypothesis \ref{2312-22hypo2} (B.3), and  Hypothesis \ref{2312-22hypo3} (C.3), we estimate the term $I_1$ as
\begin{eqnarray}\label{2212-25eq3.40}
I_1
& \leq&   \delta \E\bigg[\int_0^{T\wedge\tau_m^N} F(s)(1+\|Y_m(s)\|_\H^2)\d s\bigg] \nonumber\\
& \leq&  \delta \int_0^{T}F(s)\d s\left(1+\E\bigg[\sup_{s\in[0,T]}\|Y_m(s)\|_\H^2\bigg]\right) \nonumber\\
&\leq& C\delta,
\end{eqnarray}
where $F(s)$ defined by \eqref{2312-25eq3.3020}.

Using Hypothesis \ref{2312-22hypo1} (A.4), Hypothesis \ref{2312-22hypo3} (C.4),   we estimate the term $I_2$ as
\begin{eqnarray}\label{2212-25eq3.41}
|I_2|
& \leq&  2\E\bigg[ \bigg|\int_0^{T\wedge\tau_m^N}\bigg(\int_{0\vee(s-\delta)}^s  \chi_{\{\tau_m^N>t\}} \langle A(s,Y_m(s)),Y_m(t\wedge\tau_m^N)\rangle    \nonumber   \\
 && ~~~~~~~~~~ + \int_{0\vee(s-\delta)}^s  \chi_{\{\tau_m^N>t\}} (\Phi(s,Y_m(s)),Y_m(t\wedge\tau_m^N))\bigg) \d t\d s \bigg|\bigg]   \nonumber   \\
& \leq& 2 \E\bigg[\int_0^{T\wedge\tau_m^N}  \|A(s,Y_m(s))\|_{\V^*} \left(\int_{0\vee(s-\delta)}^s\|Y_m(t\wedge\tau_m^N)\|_\V\d t\right)\d s \nonumber   \\
 && ~~~~~~~~~~   +\int_0^{T\wedge\tau_m^N} \|\Phi(s,Y_m(s))\|_\H      \left(\int_{0\vee(s-\delta)}^s\|Y_m(t\wedge\tau_m^N)\|_\H\d t\right)\d s\bigg]  \nonumber \\
& \leq& 2\delta^{\frac{\beta-1}{\beta}}\E\bigg[\int_0^{T\wedge\tau_m^N}  \|A(s,Y_m(s))\|_{\V^*} \d s \bigg(\int_0^{T\wedge\tau_m^N}\|Y_m(t)\|_\V^\beta\d t\bigg)^{\frac{1}{\beta}} \nonumber   \\
&& ~~~~~~~~~~  +\int_0^{T\wedge\tau_m^N} \|\Phi(s,Y_m(s))\|_\H   \d s   \bigg(\int_0^{T\wedge\tau_m^N}\|Y_m(t)\|_\V^\beta\d t\bigg)^{\frac{1}{\beta}}  \bigg]  \nonumber \\
& \leq&  2\delta^{\frac{\beta-1}{\beta}}
\E\bigg[T^{\frac{1}{\beta}} \bigg( \int_0^{T\wedge\tau_m^N}  \|A(s,Y_m(s))\|_{\V^*}^{\frac{\beta}{\beta-1}} \d s\bigg)^{\frac{\beta-1}{\beta}}   \bigg(\int_0^{T\wedge\tau_m^N}\|Y_m(t)\|_\V^\beta\d t\bigg)^{\frac{1}{\beta}} \nonumber   \\
 && ~~~~~~~~~~   +\int_0^{T\wedge\tau_m^N} \|\Phi(s,Y_m(s))\|_\H  \d s    \bigg(\int_0^{T\wedge\tau_m^N}\|Y_m(t)\|_\V^\beta\d t\bigg)^{\frac{1}{\beta}}\bigg]  \nonumber \\
 %
& \leq&  C\delta^{\frac{\beta-1}{\beta}}.
\end{eqnarray}
Combining \eqref{2212-25eq3.39}-\eqref{2212-25eq3.41} yields
\begin{eqnarray*}\label{2212-25eq3.42}
&&  \sup_{m\in \mathbb{N}}\E\bigg[\int_0^{T-\delta}\|Y_m^N(t+\delta)-Y_m^N(t)\|_\H^2\d t\bigg]  \leq C(\delta+ \delta^{\frac{\beta-1}{\beta}})
\end{eqnarray*}
Now, for $\beta\in(1,2]$, we use H\"older's inequality to get
\begin{eqnarray*}\label{2212-25eq3.44}
&& \lim_{\delta\to 0+} \sup_{m\in\mathbb{N}}\E\bigg[\int_0^{T-\delta}\|Y_m^N(t+\delta)-Y_m^N(t)\|_\H^\beta\d t\bigg] \\
& \leq &  C\lim_{\delta\to0+}\sup_{m\in\mathbb{N}}\Bigg\{\E\bigg[\int_0^{T-\delta}\|Y_m^N(t+\delta)-Y_m^N(t)\|_\H^2\d t\bigg]\Bigg\}^{\frac{\beta}{2}}=0,
\end{eqnarray*}
which completes the proof of \eqref{2312-25eq3.37} for $\beta\in(1,2]$.

\

\noindent \textbf{The case of $\beta\in(2,\infty)$.}
 Again, applying It\^o's formula to the process $$\|Y_m^N(t+\delta)-Y_m^N(t)\|_\H^\beta,$$ and  then taking expectations, we find
\begin{eqnarray}\label{2212-25eq3.45}
 && \E\big[\|Y_m^N(t+\delta)-Y_m^N(t)\|_\H^\beta\big]   \nonumber \\
 &=& \frac{\beta}{2}\E\bigg[\int_{t\wedge\tau_m^N}^{(t+\delta)\wedge\tau_m^N}  \|Y_m(s)-Y_m(t\wedge\tau_m^N)\|_\H^{\beta-2}  \big[2\langle A(s,Y_m(s)),Y_m(s)-Y_m(t\wedge\tau_m^N)\rangle   \nonumber \\
 &&  +  2( \Phi(s,Y_m(s)),Y_m(s)-Y_m(t\wedge\tau_m^N)) + \|\PP_m B(s, Y_m(s))\Q_m\|_{L_2}^2\big]\d s\bigg]    \nonumber \\
 &&  +  \frac{\beta(\beta-2)}{2}\E\bigg[\int_{t\wedge\tau_m^N}^{(t+\delta)\wedge\tau_m^N}\|Y_m(s)-Y_m(t\wedge\tau_m^N)\|_{\H}^{\beta-4}   \nonumber \\
 &&  ~~~~~~~~~~~~~~~~~~~~ \times\|(Y_m(s)-Y_m(t\wedge\tau_m^N))\circ \PP_mB(s,Y_m(s))\Q_m\|_{\U}^2\d s\bigg].
\end{eqnarray}
Integrating over $[0,T-\delta]$, using Fubini's theorem, and then   we find
\begin{eqnarray}\label{2212-25eq3.46}
 && \E\bigg[\int_0^{T-\delta}\|Y_m^N(t+\delta)-Y_m^N(t)\|_\H^\beta\d t\bigg]   \nonumber \\
 &\leq&   C_{\beta}\E\bigg[\int_0^{T-\delta}\bigg\{\int_{t\wedge\tau_m^N}^{(t+\delta)\wedge\tau_m^N}\|Y_m(s)-Y_m(t\wedge\tau_m^N)\|_\H^{\beta-2} \Big[2\langle A(s,Y_m(s)),Y_m(s)\rangle  \nonumber \\
 && ~~~~~~~~~ + 2(\Phi(s,Y_m(s)),Y_m(s))  +\|\PP_mB(s,Y_m(s))\Q_m\|_{L_2}^2\Big]\d s\bigg\}\d t\bigg]    \nonumber \\
 && -\beta\E\bigg[\int_0^{T-\delta}\bigg\{\int_{t\wedge\tau_m^N}^{(t+\delta)\wedge\tau_m^N}  \|Y_m(s)-Y_m^N(t\wedge\tau_m^N)\|_\H^{\beta-2}   \nonumber \\
 && ~~~~~~~~~ \times\Big[ \langle A(s,Y_m(s)),Y_m(t\wedge\tau_m^N)\rangle  +  (\Phi(s,Y_m(s)),Y_m(t\wedge\tau_m^N) ) \Big] \d s\bigg\}\d t\bigg]   \nonumber \\
 &=:& J_1+J_2.
\end{eqnarray}
Since we are considering the case $\tau_m^N=T$ only, we know that $\mathbb{P}\left(\chi_{\{\tau_m^N\leq t\leq T-\delta\}}\right)=0$, applying Fubini's theorem, Hypothesis \ref{2312-22hypo1} (A.3), Hypothesis \ref{2312-22hypo2} (B.3),  Hypothesis \ref{2312-22hypo3} (C.3), H\"older's and Young's inequalities, similarly to $I_1$, we estimate the term $J_1$ as,
\begin{eqnarray}\label{2212-25eq3.47}
J_1
 &\leq& C \E\bigg[\int_0^{T\wedge\tau_m^N}\bigg(\delta\|Y_m(s)\|_\H^{\beta-2}  +  \delta\sup_{0\leq t\leq T\wedge\tau_m^N}\|Y_m(t)\|_\H^{\beta-2}\bigg)    \nonumber \\
 &&  ~~~~~~~~~  \times  F(s)\bigg(1+\|Y_m(s)\|_{\H}^2\bigg) \d s\bigg]   \nonumber \\
 &\leq&   C\delta\E\bigg[\int_0^{T\wedge\tau_m^N} F(s)\big(1+\|Y_m(s)\|_\H^{2}\big)\|Y_m(s)\|_\H^{\beta-2}\d s   \nonumber \\
 && ~~~~~~~~~  + \sup_{0\leq t\leq T\wedge\tau_m^N}\|Y_m(t)\|_\H^{\beta-2}\int_0^{T\wedge\tau_m^N}   F(s)\big(1+\|Y_m(s)\|_\H^{2}\big)\d s\bigg]  \nonumber \\
 &\leq&   C	\delta \E\bigg[\int_0^{T\wedge\tau_m^N} F(s)\left(1+\|Y_m(s)\|_\H^{\beta}\right)\d s\bigg]   \nonumber \\
 && +C\delta\E\bigg[\sup_{0\leq t\leq T\wedge\tau_m^N}\left\{ \|Y_m(t)\|_\H^{\beta-2}\big(1+\|Y_m(t)\|_\H^{2}\big)\right\}\int_0^{T}  F(s)\d s\bigg] \nonumber \\
 &\leq& C\delta.
\end{eqnarray}
Now, consider the term $J_2$ and we estimate it, similarly to $I_2$, using Fubini's theorem, H\"older's inequality,     Hypothesis \ref{2312-22hypo1} (A.4), and Hypothesis \ref{2312-22hypo3} (C.4) as
\begin{eqnarray}\label{2212-25eq3.48}
|J_2|
 &\leq&  \beta\E\bigg[\bigg|\int_0^{T\wedge\tau_m^N} \int_{0\vee(s-\delta)}^s  \chi_{\{\tau_m^N>t\}}\|Y_m(s)-Y_m(t\wedge\tau_m^N)\|_\H^{\beta-2}  \nonumber \\
 &&  ~~~~~~~~~ \times\Big[ \langle A(s,Y_m(s)),Y_m(t\wedge\tau_m^N) \rangle  +  (\Phi(s,Y_m(s)),Y_m(t\wedge\tau_m^N) ) \Big] \d t\bigg\}\d s\bigg]  \nonumber \\
 &\leq&  C\beta\E\bigg[\bigg|\int_0^{T\wedge\tau_m^N} \int_{0\vee(s-\delta)}^s  \chi_{\{\tau_m^N>t\}}\|Y_m(s)\|_\H^{\beta-2}  \nonumber \\
 &&  ~~~~~~~~~ \times\Big[ \langle A(s,Y_m(s)),Y_m(t\wedge\tau_m^N) \rangle  +  (\Phi(s,Y_m(s)),Y_m(t\wedge\tau_m^N) ) \Big] \d t\bigg\}\d s\bigg]  \nonumber \\
 && + C\beta\E\bigg[\bigg|\int_0^{T\wedge\tau_m^N} \int_{0\vee(s-\delta)}^s  \chi_{\{\tau_m^N>t\}}\|Y_m(t\wedge\tau_m^N)\|_\H^{\beta-2}  \nonumber \\
 &&  ~~~~~~~~~ \times\Big[ \langle A(s,Y_m(s)),Y_m(t\wedge\tau_m^N) \rangle  +  (\Phi(s,Y_m(s)),Y_m(t\wedge\tau_m^N) ) \Big] \d t\bigg\}\d s\bigg]  \nonumber \\
 &\leq&  C\beta\E\bigg[  \sup_{0\leq s \leq T\wedge \tau_m^N}\|Y_m(s)\|_\H^{\beta-2} \nonumber \\
 &&  ~~~~~~~~~  \bigg\{\int_0^{T\wedge\tau_m^N} \| A(s,Y_m(s))\|_{\V^*}  \int_{0\vee(s-\delta)}^s \| Y_m(t\wedge\tau_m^N)\|_\V  \d t \d s\nonumber \\
 &&  ~~~~~~~~~   + \int_0^{T\wedge\tau_m^N} \|\Phi(s,Y_m(s))\|_\H  \int_{0\vee(s-\delta)}^s \| Y_m(t\wedge\tau_m^N)\|_\H  \d t\d s \bigg\}\bigg] \nonumber \\
 && +  C\beta\E\bigg[  \sup_{0\leq s \leq T\wedge \tau_m^N}\|Y_m(t\wedge\tau_m^N)\|_\H^{\beta-2} \nonumber \\
 &&  ~~~~~~~~~   \bigg\{\int_0^{T\wedge\tau_m^N} \| A(s,Y_m(s))\|_{\V^*}  \int_{0\vee(s-\delta)}^s \| Y_m(t\wedge\tau_m^N)\|_\V  \d t \d s\nonumber \\
 &&  ~~~~~~~~~   + \int_0^{T\wedge\tau_m^N} \|\Phi(s,Y_m(s))\|_\H  \int_{0\vee(s-\delta)}^s \| Y_m(t\wedge\tau_m^N)\|_\H  \d t\d s \bigg\}\bigg]            \nonumber \\
 &\leq&    C\beta\delta^{\frac{\beta-1}{\beta}}\E\bigg[  \sup_{0\leq s \leq T\wedge \tau_m^N}\|Y_m(s)\|_\H^{\beta-2} \nonumber \\
 &&  ~~~~~~~~~  \bigg\{\int_0^{T\wedge\tau_m^N} \| A(s,Y_m(s))\|_{\V^*}  \d s \Big( \int_{0}^{T\wedge\tau_m^N} \| Y_m(t)\|_\V^\beta  \d t\Big)^{\frac{1}{\beta}}\nonumber \\
 &&  ~~~~~~~~~   + \int_0^{T\wedge\tau_m^N} \|\Phi(s,Y_m(s))\|_\H   \d s \Big( \int_{0}^{T\wedge\tau_m^N} \| Y_m(t)\|_\V^\beta  \d t\Big)^{\frac{1}{\beta}} \bigg\}\bigg] \nonumber \\
 && +  C\beta\delta^{\frac{\beta-1}{\beta}} \E\bigg[  \sup_{0\leq s \leq T\wedge \tau_m^N}\|Y_m(t\wedge\tau_m^N)\|_\H^{\beta-2} \nonumber \\
 &&  ~~~~~~~~~  \bigg\{\int_0^{T\wedge\tau_m^N} \| A(s,Y_m(s))\|_{\V^*}   \d s \Big( \int_{0}^{T\wedge\tau_m^N} \| Y_m(t)\|_\V^\beta  \d t\Big)^{\frac{1}{\beta}}\nonumber \\
 &&  ~~~~~~~~~   + \int_0^{T\wedge\tau_m^N} \|\Phi(s,Y_m(s))\|_\H   \d s \Big( \int_{0}^{T\wedge\tau_m^N} \| Y_m(t)\|_\V^\beta  \d t\Big)^{\frac{1}{\beta}} \bigg\}\bigg]            \nonumber \\
 &\leq&    C\delta^{\frac{\beta-1}{\beta}}.
\end{eqnarray}
Substituting \eqref{2212-25eq3.47} and \eqref{2212-25eq3.48} in \eqref{2212-25eq3.46}, we arrive at
\begin{equation*}\label{2212-25eq3.49}
  \sup_{m\in\mathbb{N}}\E\bigg[\int_0^{T-\delta}\|Y_m(s)-Y_m(t\wedge\tau_m^N)\|_\H^\beta\bigg] \leq C(\delta+\delta^{\frac{\beta-1}{\beta}}).
\end{equation*}
Hence, our desired \eqref{2312-25eq3.37} follows for $\beta\geq 2$.
\end{proof}

Our goal in this section is to identify a suitable continuous solution, so  we need to prove the tightness of the family $\{\mathcal{L}(Y_m):m\in\mathbb{N}\}$ on the space of continuous functions, for more details, see Billingsley (2013) or  M\'{e}tivier  (1988). To achieve this, we must verify the following Aldous condition, see Aldous (1978) or Theorem 3.2 in  M\'{e}tivier  (1988).
\begin{definition}
	Let $(\mathbb{Y},\|\cdot\|_{\mathbb{Y}})$ be a separable Banach space and let $\{Y_m\}$ be a sequence of $\mathbb{Y}$-valued random variables. Assume that for every $\epsilon,\eta>0$, there is a $\delta>0$ such that for every sequence  $\{\tau_m\}_{m\in\mathbb{N}}$  of $\mathcal{F}$-stopping times with $\tau_m\leq T$,  the following inequality holds:
\begin{align}
   \sup_{m\in\mathbb{N}}\sup_{0<\xi\leq \delta}\P\left\{\|Y_m((\tau_m+\xi)\wedge T)-Y_m(\tau_m)\|_{\mathbb{Y}}\geq \eta\right\}\leq \epsilon.
\end{align}
In this case, we say that the sequence $\{Y_m\}$  satisfies the Aldous condition.
\end{definition}

If a sequence $\{Y_m\}$ satisfies the Aldous condition in  $\mathbb{Y}$, then  the laws of $\{Y_m\}$ form a tight sequence on $C([0,T];\mathbb{Y})$.
By applying the Chebyshev inequality, we can establish a sufficient condition for verifying the Aldous condition.

\begin{proposition}\label{2312-22lemAldous}
Given any $\varepsilon>0$,  there is a $\delta>0$ such that
\begin{align}\label{2312-25eq3.50}
	\sup_{m\in\mathbb{N}}\sup_{0<\xi\leq \delta}\E\big[\|Y_m((\tau_m+\xi)\wedge T)-Y_m(\tau_m)\|_{\mathbb{Y}}^\zeta\big]\leq \varepsilon,
\end{align}
for some $\zeta>0$, then the sequence $\{Y_m\}$ satisfies the Aldous condition in the space $\mathbb{Y}$. 
\end{proposition}

Now, our aim is to prove that the laws of sequence of approximated solutions $\{Y_m\}$ denoted by $\mathcal{L}(Y_m)$ are tight  as a probability measure on the space $C([0,T];\V^*)\cap L^\beta(0,T;\H)$.
We have already proved the set of measures $\{\mathcal{L}(Y_m):m\in\mathbb{N}\}$ is tight in $L^\beta(0,T;\H)$ in Lemma \ref{2312-22lem2}.
Next, we will verify the set of measures $\{\mathcal{L}(Y_m):m\in\mathbb{N}\}$ is tight in $C([0,T];\V^*)$, for this, it is enough to establish \eqref{2312-25eq3.50}, where $\mathbb{Y}$ is selected as $\V^*$.

\begin{lemma}\label{2312-22lem3.5}
Assume that $x\in\H$. Then the laws $\{\mathcal{L}(Y_m):m\in \mathbb{N})\}$ 
 form a tight sequence of probability measures on $C([0,T];\V^*)$.
\end{lemma}
\begin{proof}
  Let us consider stopping times $\{\tau_m\}$ such that $0\leq \tau_m\leq T$. It is follows from \eqref{2312-25eq3.17} that
\begin{eqnarray}\label{2312-25eq3.51}
 && Y_m((\tau_m+\xi)\wedge T)-Y_m(\tau_m\wedge T)   \nonumber \\
 &= & \int_{\tau_m\wedge T}^{(\tau_m+\xi)\wedge T}\PP_m A(s, Y_m(s))\d s  +\int_{\tau_m\wedge T}^{(\tau_m+\xi)\wedge T}  \PP_m \Phi(s, Y_m(s)) \d s    \nonumber \\
 &&    +\int_{\tau_m\wedge T}^{(\tau_m+\xi)\wedge T}  \PP_m B(s, Y_m(s))\Q_m\d W(s).
\end{eqnarray}

Applying the basic $C_r$-inequality, taking expectation on \eqref{2312-25eq3.51},  and then by the BDG inequality, we get
\begin{eqnarray}\label{2312-25eq3.52}
 && \E\left[\| Y_m((\tau_m+\xi)\wedge T)-Y_m(\tau_m) \|_{\V^*}^\beta\right]  \nonumber \\
 &\leq & C \E \bigg[\bigg(\int_{\tau_m\wedge T}^{(\tau_m+\xi)\wedge T} \|\PP_m A(s, Y_m(s))\|_{\V^*}\d s\bigg)^\beta  \bigg]   \nonumber \\
 &&   +  C \E \bigg[\bigg(\int_{\tau_m\wedge T}^{(\tau_m+\xi)\wedge T}  \|\PP_m \Phi(s, Y_m(s))\|_{\H} \d s \bigg)^\beta\bigg]   \nonumber \\
 &&    + C \E \bigg[\bigg( \int_{\tau_m\wedge T}^{(\tau_m+\xi)\wedge T}  \|\PP_m B(s, Y_m(s))\Q_m\|_{L_2}^2 \d s \bigg)^{\frac\beta2}\bigg]  \nonumber \\
 &=:&  J_1+J_2+J_3.
\end{eqnarray}
Firstly, we consider the term $J_1$ and estimate it using Hypothesis \ref{2312-22hypo1} (A4) and H\"{o}lder's inequality as
\begin{eqnarray}\label{2312-25eq3.53}
 && \E \bigg[\bigg(\int_{\tau_m\wedge T}^{(\tau_m+\xi)\wedge T} \|\PP_m A(s, Y_m(s))\|_{\V^*}\d s\bigg)^\beta  \bigg]  \nonumber \\
 &\leq & C\xi^{\frac1\beta} \E \bigg[ \int_{\tau_m\wedge T}^{(\tau_m+\xi)\wedge T} \|  A(s, Y_m(s))\|_{\V^*}^{\frac{\beta}{\beta-1}} \d s   \bigg]^{\beta-1}   \nonumber \\
 &\leq & C\xi^{\frac1\beta} \E \bigg[ \int_{\tau_m\wedge T}^{(\tau_m+\xi)\wedge T}  (f_A(s)+C\|Y_m(s)\|_\V^\beta)(1+\|Y_m(s)\|_\H^\alpha) \d s \bigg]^{\beta-1}   \nonumber \\
 &\leq&    C\xi^{\frac1\beta},
\end{eqnarray}
where we have used the fact that $f_A\in L^1(0,T; \R_+)$ and \eqref{2312-25eq3.18} in Lemma \ref{2312-25lem1}.

Secondly, we consider the term $J_2$ and we estimate it using  Hypothesis \ref{2312-22hypo3} (C4), H\"{o}lder's inequality, and Poincar\'{e} inequality (see Evans (2010)), as
\begin{eqnarray}\label{2312-25eq3.54}
 && \E \bigg[\bigg(\int_{\tau_m\wedge T}^{(\tau_m+\xi)\wedge T} \|\PP_m \Phi(s, Y_m(s))\|_{\H}\d s\bigg)^\beta  \bigg]  \nonumber \\
 &\leq &   \E \bigg[ \bigg( \int_{\tau_m\wedge T}^{(\tau_m+\xi)\wedge T}  \|\Phi(s, Y_m(s))\|_{\H}\d s\bigg)^\beta  \bigg]  \nonumber \\
 &\leq & C  \E \bigg[ \bigg( \int_{\tau_m\wedge T}^{(\tau_m+\xi)\wedge T}   (f_\Phi(s)+\|Y_m(s)\|_\H)    \d s\bigg)^\beta  \bigg]  \nonumber \\
 &\leq & C  \E \bigg[ \bigg( \int_{\tau_m\wedge T}^{(\tau_m+\xi)\wedge T}  (f_\Phi(s)+\|Y_m(s)\|_\V)    \d s\bigg)^\beta  \bigg]  \nonumber \\
 &\leq &  C  \E \bigg[ \bigg( \int_{\tau_m\wedge T}^{(\tau_m+\xi)\wedge T}   f_\Phi(s)     \d s\bigg)^\beta  + \bigg( \int_{\tau_m\wedge T}^{(\tau_m+\xi)\wedge T}  \|Y_m(s)\|_\V    \d s\bigg)^\beta  \bigg]  \nonumber \\
 &\leq &  C  \E \bigg[ \bigg( \int_{\tau_m\wedge T}^{(\tau_m+\xi)\wedge T}   f_\Phi(s)     \d s\bigg)^\beta  +  \int_{\tau_m\wedge T}^{(\tau_m+\xi)\wedge T}  \|Y_m(s)\|_\V^\beta    \d s   \bigg].
\end{eqnarray}
Noting that $f_\Phi\in L^1(0,T; \R_+)$ and  \eqref{2312-25eq3.18} in Lemma \ref{2312-25lem1}, by the absolute continuity of the Lebesgue integral, we get the existence of an $\varepsilon_1$ such that
\begin{eqnarray}\label{2312-25eq3.56.1}
 && \sup_{m\in \mathbb{N}} \sup_{0<\xi\leq \delta} \E \bigg[\bigg(\int_{\tau_m\wedge T}^{(\tau_m+\xi)\wedge T} \|\PP_m \Phi(s, Y_m(s))\|_{\H}\d s\bigg)^\beta  \bigg]
 \leq C \varepsilon_1.
\end{eqnarray}

Now, we consider the term $J_3$ and we estimate it using  Hypothesis \ref{2312-22hypo2} (B3), H\"{o}lder's inequality,  and It\^{o} isometry as
\begin{eqnarray}\label{2312-25eq3.55}
 && \E \bigg[\bigg( \int_{\tau_m\wedge T}^{(\tau_m+\xi)\wedge T}  \|\PP_m B(s, Y_m(s))\Q_m\|_{L_2}^2 \d s \bigg)^{\frac\beta2}\bigg]  \nonumber \\
 &\leq &  \E \bigg[\bigg( \int_{\tau_m\wedge T}^{(\tau_m+\xi)\wedge T}  g_B(s)(1+\|Y_m(s)\|_\H^2) \d s \bigg)^{\frac\beta2}\bigg]   \nonumber \\
 &\leq & C \E \bigg[\bigg( \int_{\tau_m\wedge T}^{(\tau_m+\xi)\wedge T}  g_B(s)  \d s \bigg)^{\frac\beta2}\bigg].
\end{eqnarray}
Noting that   $g\in L^1(0,T;\R_+)$, by the absolute continuity of the Lebesgue integral, we get the existence of an $\varepsilon_2$ such that
\begin{eqnarray}\label{2312-25eq3.56.2}
 && \sup_{m\in \mathbb{N}} \sup_{0<\xi\leq \delta}  \E \bigg[\bigg( \int_{\tau_m\wedge T}^{(\tau_m+\xi)\wedge T}  \|\PP_m B(s, Y_m(s))\Q_m\|_{\H}^2 \d s \bigg)^{\frac\beta2}\bigg]
 \leq C \varepsilon_2.
\end{eqnarray}
Combining the estimates \eqref{2312-25eq3.53}-\eqref{2312-25eq3.56.2}, one can conclude that the family $\{\mathcal{L}(Y_m)\}$ is tight in the space $C(0,T; \V^*)$.
\end{proof}

Set $\Gamma := C([0,T];\V^*)\cap L^\beta(0,T;\H)\times C([0,T];\U_1)$, 
where $\U_1$ is a Hilbert space such that the embedding $\U\subset\U_1$ is Hilbert-Schmidt.
In view of  Lemma \ref{2312-22lem2} and Lemma \ref{2312-22lem3.5}, for $W_m:=W$, $m\in\mathbb{N}$, we find that the family of the laws $\mathcal{L}(Y_m,W_m)$ of the random vectors $(Y_m,W_m)$ is tight in $\Gamma$.
From the Prokhorov's theorem  and a version of Skorokhod's representation theorem, see Theorem C.1 in Brzezniak and Hausenblas (2018) or Theorem A.1 in Nguyen, Tawri and Temam (2021), we can construct a new probability space $(\wi{\Omega},\wi{\mathcal{F}},\wi{\P})$ and a subsequence of random vectors $\{(\wi{Y}_m,\wi{W}_m)\}$ (still denoted by  the same)  and $(\wi{Y},\wi{W})$ on the space $\Gamma$  such that
\begin{enumerate}
	\item $\mathcal{L}(\wi{Y}_m,\wi{W}_m)=\mathcal{L}(\wi{Y},\wi{W})$, for all $m\in\mathbb{N}$;
	\item $(\wi{Y}_m,\wi{W}_m) \to (\wi{Y},\wi{W})$ in $\Gamma$ with probability $1$ on probability space  $(\wi{\Omega},\wi{\mathcal{F}},\wi{\P})$ as $m\to\infty$;
	\item $(\wi{W}(\wi{\omega}))=(W(\wi{\omega}))$ for all $\wi{\omega}\in\wi{\Omega}$;
\end{enumerate}
Using the definition of $\Gamma$,  we have
\begin{eqnarray}\label{2312-25eq3.57}
	\|\wi{Y}_m-\wi{Y}\|_{L^\beta(0,T;\H)}+\|\wi{Y}_m-\wi{Y}\|_{C([0,T];\V^*)}    \to   0, \ \wi{\P} \text{-a.s.}
\end{eqnarray}

Our next goal is to prove that $(\wi{Y},\wi{W})$ is a solution to the system  (\ref{2312-22eq01.1}).
 Let us denote the filtration  by $\{\wi{\mathcal{F}}_t\}_{t\geq 0}$ satisfying the usual conditions and generated by $\{\wi{Y}_m(s),\wi{Y}(s),\wi{W}(s):s\leq t\}$.

 Then, $\wi{W}$ is an $\{\wi{\mathcal{F}}_t\}$-cylindrical Wiener process on $\U$.  The equation \eqref{2312-25eq3.17} satisfied by the random vector $(Y_m,W_m)=(Y_m,W)$,  and hence it follows that
 \begin{eqnarray}\label{2312-25eq3.58}
 	\wi{Y}_m(t)
 &=&\PP_m x  +\int_{0}^{t}\PP_m A(s,\wi{Y}_m(s))\d s    +\int_{0}^{t}\PP_m \Phi(s,\wi{Y}_m(s))\d s  \nonumber	\\
 & &+\int_0^t\PP_m B(s,\wi{Y}_m(s))\Q_m\d\wi{W}(s).
 \end{eqnarray}
Following the Lemma \ref{2312-25lem1}, $\wi{Y}_m(t)$ also satisfy the  energy estimate, that is, for any $p\geq2,$  there is a constant $C_p>0$ such that
\begin{eqnarray}\label{2312-25eq3.59}
	\sup_{m\in\N}\left\{\wi{\E}\bigg[\sup_{0\leq t\leq T}\|\wi{Y}_m(t)\|_\H^p\bigg]  +  \wi{\E}\bigg[\int_0^T\|\wi{Y}_m(t)\|_\V^\beta\d t\bigg]^{\frac{p}{2}}\right\}
< C_p(1+\|x\|_\H^p).
\end{eqnarray}
Using the fact that $\|\cdot\|_{\H}$ and $\|\cdot\|_{\V}$  are lower semicontinuous in $\V^*$, the convergence \eqref{2312-25eq3.57} and Fatou's lemma yield
\begin{eqnarray}\label{2312-25eq3.60}
&&	\wi{\E}\bigg[\sup_{0\leq t\leq T}\|\wi{Y}(t)\|_\H^p\bigg]
    \leq \wi{\E}\bigg[\sup_{0\leq t\leq T}\liminf_{m\to\infty}\|\wi{Y}_m(t)\|_\H^p\bigg] \nonumber \\
&\leq& \wi{\E}\bigg[\liminf_{m\to\infty}\sup_{0\leq t\leq T}\|\wi{Y}_m(t)\|_\H^p\bigg]
     \leq  \liminf_{m\to\infty}	\wi{\E}\bigg[\sup_{0\leq t\leq T}\|\wi{Y}_m(t)\|_\H^p\bigg]  \nonumber \\
&<&C_p(1+\|x\|_\H^p).
\end{eqnarray}
Similarly, from \eqref{2312-25eq3.59}, one can also deduce that
\begin{eqnarray}\label{2312-25eq3.61}
	\wi{\E}\bigg[\int_0^T\|\wi{Y}(t)\|_\V^\beta\d t\bigg]^{\frac{p}{2}} < C_p(1+\|x\|_\H^p).
\end{eqnarray}

By virtue of \eqref{2312-25eq3.59}, and using Hypothesis \ref{2312-22hypo1} (A.4), Hypothesis \ref{2312-22hypo2} (B.3) and Hypothesis \ref{2312-22hypo3} (C.4), the following estimates hold.

\begin{lemma}\label{2312-22lem3.3}
	The following estimates hold:
	\begin{eqnarray}
		\sup_{m\in\N}\wi{\E}\bigg[\int_0^T\|\PP_m A(t,\wi{Y}_m(t))\|_{\V^*}^{\frac{\beta}{\beta-1}}\d t\bigg]  &<&\infty,  \label{2312-25eq3.62}\\
		\sup_{m\in\N}\wi{\E}\bigg[\int_0^T\|\PP_m\Phi(t,\wi{Y}_m(t))\|_{\H}^2\d t\bigg]                    &<&\infty, \label{2312-25eq3.63}\\
		 \sup_{m\in\N}\wi{\E}\bigg[\int_0^T\|\PP_mB(t,\wi{Y}_m(t))\Q_m\|_{L_2}^2\d t\bigg]               &<&\infty.  \label{2312-25eq3.64}
	\end{eqnarray}
\end{lemma}

Applying the well-known Banach-Alaoglu theorem, by the above Lemma \ref{2312-22lem3.3}, we obtain the existence of
$\overline{Y}\in L^\beta(\wi{\Omega}\times[0,T];\V)$,
$\wi{A}\in L^{\frac{\beta}{\beta-1}}(\wi{\Omega}\times[0,T];\V^*)$,
$\wi{\Phi}\in L^{2}(  [0,T];\H)$,
$\wi{B}\in L^2(\wi{\Omega}\times[0,T];L_2(\U,\H))$  such that along a subsequence the following convergence hold, as $m\to\infty$:
\begin{eqnarray*}
 \wi{Y}_m      &\xrightarrow{w} &  \overline{Y}, ~~~ \text{ in } L^\beta(\wi{\Omega}\times[0,T];\V),    \label{2312-25eq3.65} \\
 \PP_m	A(\cdot,\wi{Y}_m(\cdot))  &\xrightarrow{w}& \wi{A}(\cdot),  \text{ in } L^{\frac{\beta}{\beta-1}}(\wi{\Omega}\times[0,T];\V^*),   \label{2312-25eq3.66}\\
 \PP_m	\Phi(\cdot,\wi{Y}_m(\cdot))  &\xrightarrow{w}& \wi{\Phi}(\cdot),  \text{ in } L^{2}(\wi{\Omega}\times[0,T];\H),   \label{2312-25eq3.67}\\
 \PP_m  B(\cdot,\wi{Y}_m(\cdot))\Q_m   &\xrightarrow{w}&   \wi{B}(\cdot), \text{ in } L^2(\wi{\Omega}\times[0,T];L_2(\U,\H)),    \label{2312-25eq3.68} \\
 \int_0^\cdot\PP_m B(s,\wi{Y}_m(s))\Q_m\d\wi{W}(s)   &\xrightarrow{w}&    \int_0^\cdot\wi{B}(s)\d\wi{W}(s),\text{ in } L^\infty(0,T;L^2(\wi{\Omega};\H)), \label{2312-25eq3.69}
\end{eqnarray*}
 Let us set
\begin{eqnarray}\label{2312-25eq3.70}
	\tilde{Y}(t):= x + \int_0^t\wi{A}(s)\d s + \int_0^t\wi{\Phi}(s)\d s  +  \int_0^t\wi{B}(s)\d\wi{W}(s).
\end{eqnarray}
Then, one can verify that
\begin{equation}\label{2312-25eq3.71}
 \wi{Y}=\overline{Y}=\tilde{Y},\; \wi{\P}\otimes\d t\text{-a.e.},
\end{equation}
where the first equality in \eqref{2312-25eq3.71} holds from the uniqueness of the limits. Moreover, by Theorem 4.2.5 in Liu and R\"{o}ckner (2015), we also know that $\tilde{Y}$ is an $\H$-valued continuous process. In view of \eqref{2312-25eq3.60}, $\wi{Y}$ is $\H$-valued, and by its continuity in $\V^*$,  $\wi{Y}$ is weakly continuous in $\H$. Therefore, $\tilde{Y}$ and $\wi{Y}$ are indistinguishable.

In the sequel, we prove our results in the newly constructed filtered probability space
$$(\wi{\Omega}, \wi{\mathcal{F}}, \{\wi{\mathcal{F}}_t\}_{t\geq 0}, \wi{\P}).$$
Now,  we drop the superscript notation, for example, we  write $\{\wi{Y}_m\}$ and  $\wi{Y}$ as $\{Y_m\}$ and $Y,$ respectively. Therefore, \eqref{2312-25eq3.57} can be rewritten as	\begin{align*}
	\|Y_m-Y\|_{L^\beta(0,T;\H)}+\|Y_m-Y\|_{C([0,T];\V^*)}  \to 0.
\end{align*}

Now, we recall some convergence  results from R\"{o}ckner, Shang and Zhang (2024).
\begin{lemma}\label{2312-22lem3.12}
\begin{description}
  \item[(i)] $\wi{B}(\cdot)=B(\cdot,Y(\cdot)), ~~ \P\otimes \d t$-a.e.
  \item[(ii)] Assume that the Hypothesis \ref{2312-22hypo1} (H.1) and (H.2)$'$ hold, the embedding $\V\subset\H$ is compact. Then $A(t,\cdot)$ is pseudo-monotone from $\V\to\V^*$ for any $t\in[0,T]$.
  \item[(iii)] If
\begin{eqnarray}\label{2312-25eq3.72}
	Y_m                &\xrightarrow{w}&        Y, ~~~~   \text{ in } \   L^\beta(\Omega\times[0,T];\V),\nonumber\\
	A(\cdot,Y_m(\cdot))       &\xrightarrow{w} & \wi{A}(\cdot),\  \text{ in }\    L^{\frac{\beta}{\beta-1}}(\Omega\times[0,T];\V^*),
\end{eqnarray}
and
\begin{eqnarray}
\liminf_{m\to\infty}\E\bigg[\int_0^T\langle A(t,Y_m(t)),Y_m(t)\rangle\d t\bigg]   &\geq&  \E\bigg[\int_0^T\langle\wi{ A}(t),Y(t)\rangle\d t\bigg], \label{2312-25eq3.73}
\end{eqnarray}
then $\wi{A}(\cdot)=A(\cdot,Y(\cdot)),~~~~~~~ \P\otimes \d t$-a.e.
\end{description}
\end{lemma}
\begin{proof}
The proof of (i)-(iii) in  this lemma can be founded in Lemma 2.14-Lemma 2.16 in  R\"{o}ckner, Shang and Zhang (2024).
\end{proof}

\begin{lemma}\label{2312-22lem3.10}
	$\wi{\Phi}(\cdot) =\Phi(\cdot,Y(\cdot)),\;\P\otimes \d t$-a.e.
\end{lemma}
\begin{proof}
	We know that $\|Y_m-Y\|_{L^\beta(0,T;\H)}\to0,\ \P$-a.s. Using   \eqref{2312-25eq3.59} and \eqref{2312-25eq3.60},  and Vitali's convergence theorem, we find
\begin{eqnarray}\label{2312-25eq3.74}
	\lim_{m\to\infty}\E\bigg[\int_0^T\|Y_m(t)-Y(t)\|_\H^\rho\d t\bigg]=0,  \ \text{ for all } \ \rho \in[1,\beta],
\end{eqnarray}
since for all $1<p<\infty$
\begin{eqnarray*}
	\E\left[\left(\int_0^T\|Y_m(t)\|_{\H}^{\beta}\d t\right)^p\right]
\leq  T^p\E\left[\sup_{t\in[0,T]}\|Y_m(t)\|_{\H}^{p\beta}\right]
<\infty.
\end{eqnarray*}
Therefore, along a subsequence $\{Y_m\}$ (still denoting by the same index), we have the following convergence:
\begin{eqnarray}\label{2312-25eq3.75}
	\lim_{m\to\infty} \|Y_m(t,\omega)-Y(t,\omega)\|_\H=0,\text{ a.e. } (t,\omega).
\end{eqnarray}
Now,   we can obtain
\begin{eqnarray}\label{2312-25eq3.76}
 && \E\bigg[\int_0^T \|\PP_m\Phi(t,Y_m(t))-\Phi(t,Y(t))\|_\H^2 \d t\bigg]   \nonumber\\
 &\leq&   \E\bigg[\int_0^T \|\PP_m\Phi(t,Y_m(t))- \PP_m\Phi(t,Y(t))\|_\H^2 \d t   \nonumber\\
 && ~~~~  +    \int_0^T \| \Phi(t,Y(t)) -  \PP_m   \Phi(t,Y(t))\|_\H^2 \d t  \bigg]   \nonumber\\
 &\leq&    \E\bigg[\int_0^T \|\PP_m(\Phi(t,Y_m(t))- \Phi(t,Y(t))) \|_\H^2 \d t   \nonumber\\
 &&    +   \E\bigg[  \int_0^T \| (I-  \PP_m)  \Phi(t,Y(t))\|_\H^2 \d t  \bigg].
\end{eqnarray}
Due to \eqref{2312-25eq3.75}, using Hypothesis \ref{2312-22hypo1} (C.2), we have
\begin{eqnarray}\label{2312-25eq3.77}
	\lim_{m\to\infty}  \E\bigg[\int_0^T \|\PP_m(\Phi(t,Y_m(t))- \Phi(t,Y(t))) \|_\H^2 \d t =0.
\end{eqnarray}
It follows from the Lebesgue dominated convergence theorem and $\lim_{m\to\infty}\| I-  \PP_m \|=0$ that
\begin{eqnarray}\label{2312-25eq3.78}
	 \lim_{m\to\infty} \E\bigg[  \int_0^T \| (I-  \PP_m)  \Phi(t,Y(t))\|_\H^2 \d t  \bigg]  =0.
\end{eqnarray}
Substituting \eqref{2312-25eq3.77}-\eqref{2312-25eq3.78} into \eqref{2312-25eq3.76}, we obtain
\begin{eqnarray}\label{2312-25eq3.79}
 && \lim_{m\to\infty}\E\bigg[\int_0^T \|\PP_m\Phi(t,Y_m(t))-\Phi(t,Y(t))\|_\H^2 \d t\bigg]  =0.
\end{eqnarray}
Combining \eqref{2312-25eq3.78} with \eqref{2312-25eq3.70}, by the uniqueness of limit, we can conclude our desired result.
\end{proof}

In the following theorem, we will establish that the limit $Y$ of the approximating sequence $\{Y_m\}$ obtained above is a probabilistically weak solution to the system (\ref{2312-22eq01.1}).

\begin{theorem}\label{2312-22TH3.11}
  There exists a probabilistically weak solution to system  (\ref{2312-22eq01.1}), which satisfies the following energy estimate
\begin{eqnarray}\label{2312-25eq3.80}
 \E\bigg[\sup_{0\leq t\leq T}\|Y(t)\|_{\H}^p\bigg]+\E\bigg[\bigg(\int_{0}^{T}\|Y(t)\|_{\V}^\beta\d t\bigg)^{\frac{p}{2}}\bigg]  \leq  C_p(1+\|x\|_{\H}^p).
\end{eqnarray}
\end{theorem}

\begin{proof}
  We will show that the limit $Y$ of $\{Y_m\}$ obtained above is a solution to  system  (\ref{2312-22eq01.1}). In order to prove this , we only need to verify  \eqref{2312-25eq3.72}.
  Taking into account the equations \eqref{2312-25eq3.58} and \eqref{2312-25eq3.70} satisfied respectively by $Y_m$ and $Y$, applying It\^{o}'s formula and taking expectations separately we get
 \begin{eqnarray}\label{2312-25eq3.81}
   \E[\|Y_m(t)\|_\H^2]
 &=& \|\PP_m x\|_\H^2  + 2 \E\bigg[\int_0^T \Big[\langle A(t, Y_m(t)), Y_m(t) \rangle   +  (\Phi(t, Y_m(t)), Y_m(t)) \Big] dt \bigg] \nonumber\\
 && + \E\bigg[ \int_0^T \|\PP_m  B(t, Y_m(t)) \Q_m\|^2_{L_2} dt \bigg],
\end{eqnarray}
  and
 \begin{eqnarray}\label{2312-25eq3.82}
   \E[\|Y(t)\|_\H^2]
 &=& \| x\|_\H^2  + 2 \E\bigg[\int_0^T \Big[\langle \wi A(t), Y(t) \rangle   +  (\wi\Phi(t), Y(t)) \Big] dt \bigg] \nonumber\\
 && + \E\bigg[ \int_0^T \| \wi B(t) \Q_m\|^2_{L_2} dt \bigg].
\end{eqnarray}
  Using the convergence \eqref{2312-25eq3.57}, the lower semi-continuity  of $\|\cdot\|_\H$ in $\V^*$ and Fatou's lemma, we obtain
  \begin{eqnarray}\label{2312-25eq3.83}
    \E[\|Y(t)\|_\H^2] \leq \E[\liminf_{m\to\infty} \|Y_m(t)\|_\H^2 ]  \leq \liminf_{m\to\infty} \E[ \|Y_m(t)\|_\H^2 ].
\end{eqnarray}
  Using the Lemma \ref{2312-22lem3.12}   and Lemma \ref{2312-22lem3.10}, and comparing \eqref{2312-25eq3.81} and \eqref{2312-25eq3.82}, we ensure that \eqref{2312-25eq3.72} holds.   Furthermore, the energy estimate \eqref{2312-25eq3.80} for $Y$ follows from \eqref{2312-25eq3.60} and \eqref{2312-25eq3.61}.
\end{proof}

\begin{theorem}\label{2312-22TH3.12}
	Under the assumptions  (A.2), (B.1) and (C.2) in Hypothesis \ref{2312-22hypo1}-\ref{2312-22hypo3}, the pathwise uniqueness holds for the solutions to the system  (\ref{2312-22eq01.1}).
	\end{theorem}
\begin{proof}
	Let us assume $Y_1(\cdot)$ and $Y_2(\cdot)$ be the two solutions to the system  (\ref{2312-22eq01.1})  defined on the same probability space $(\Omega,\mathcal{F},\{\mathcal{F}_t\}_{t\geq0},\P)$, with the initial data $Y_1(0)=x_1$ and $Y_2(0)=x_2,$ respectively. Let us define
\begin{eqnarray}\label{2312-25eq3.84}
\varphi(t):= \exp\bigg(-\int_0^t\Big[2\big(f(s)+\rho(Y_1(s))+\eta(Y_2(s))\big) + f_\Phi(t)\Big]\d s\bigg).
\end{eqnarray}
Applying It\^o's formula to the process $\varphi(\cdot)\|Y_1(\cdot)-Y_2(\cdot)\|_\H^2$, we find
\begin{eqnarray}\label{2312-25eq3.85}\nonumber
&&	\varphi(t)\|Y_1(t)-Y_2(t)\|_\H^2 \nonumber\\
&=& \|x_1-x_2\|_\H^2+\int_0^t\varphi(s)\Big\{2\langle A(s,Y_1(s))-A(s,Y_2(s)),Y_1(s)-Y_2(s)\rangle \nonumber\\
&&  \qquad + 2( \Phi(s,Y_1(s))-\Phi(s,Y_2(s)),Y_1(s)-Y_2(s)) +\|B(s,Y_1(s))-B(s,Y_2(s))\|_{L_2}^2 \nonumber \\
&& \qquad-\big[f(s)+\rho(Y_1(s))+\eta(Y_2(s))\big]\|Y_1(s)-Y_2(s)\|_\H^2\Big\}\d s  \nonumber\\
&& \quad+2\int_0^t\varphi(s)\big((B(s,Y_1(s))-B(s,Y_2(s)))\d W(s),Y_1(s)-Y_2(s)\big)  \nonumber\\
&\leq &  \|x_1-x_2\|_\H^2+2\int_0^t\varphi(s)\big((B(s,Y_1(s))-B(s,Y_2(s)))\d W(s),Y_1(s)-Y_2(s)\big),
\end{eqnarray}where we have used Hypothesis \ref{2312-22hypo1} (A.2), Hypothesis \ref{2312-22hypo2} (B.1), and Hypothesis \ref{2312-22hypo3} (C.2). Let   $\{\sigma_k\}\uparrow\infty$ be a sequence of stopping times in such a way that the local martingale  appearing in the above inequality \eqref{2312-25eq3.85} is a martingale. Taking expectations on both side of inequality \eqref{2312-25eq3.85}, we get
\begin{align}\label{2312-25eq3.86}
	\E\big[\varphi(t\wedge \sigma_k)\|Y_1(t\wedge \sigma_k)-Y_2(t\wedge \sigma_k)\|_\H^2\big] \leq \|x_1-x_2\|_\H^2.
\end{align}Passing $k\to\infty$ and using Fatou's lemma, we find
\begin{align}\label{2312-25eq3.87}
	\E\big[\varphi(t)\|Y_1(t)-Y_2(t)\|_\H^2\big] \leq \|x_1-x_2\|_\H^2,
\end{align}where we have used the fact that
\begin{align}\label{2312-25eq3.88}
	\int_0^T\big[f(s)+\rho(Y_1(s))+\eta(Y_2(s))\big]\d s<\infty,\;\P\text{-a.s.}
\end{align}The inequality \eqref{2312-25eq3.87} gives the pathwise uniqueness of solutions to the system  (\ref{2312-22eq01.1}).
\end{proof}

\begin{proof}[Proof of Theorem \ref{2312-22TH2.11}]
Applying the classical Yamada-Watanabe theorem,    Theorem \ref{2312-22TH3.11} and Theorem \ref{2312-22TH3.12} yields our desired Theorem \ref{2312-22TH2.11}.
\end{proof}

\section{Existence of optimal controls}
\setcounter{equation}{0}

The following lemma proves that given a minimizing sequence for the control problem $(\mathcal{P})$, we can obtain a subsequence and a mapping $\Phi\in \mathcal{U}$, such that the subsequence converges weakly to $\Phi$.
\begin{lemma}\label{2212-22lem4.1}
Let $\Phi_n$ be a minimising sequence for  control problem $(\mathcal{P})$. There exists a subsequence $n_k$ of $n$ and a mapping $\Phi\in \mathcal{U}$ such that for all $t\in [0,T], \ x,y\in \H$, we have
\begin{equation}\label{2312-25eq4.1}
  \lim_{k\to\infty}  (\Phi_{n_k}(t,x),  y)  = (\Phi(t,x),  y).
\end{equation}
\end{lemma}
\begin{proof}
The proof is omitted here. We refer the readers to Lemma 4.1 in Lisei (2002).
\end{proof}

For the sake of simplicity, the subsequence of $\{\Phi_{n_k}\}_{k=1}^\infty$ obtained in the previous lemma will be relabelled as the same. For this sequence and $\Phi$ as in the last lemma, let us consider the following controlled stochastic system
\begin{eqnarray}\label{2312-25eq4.2}
  (\hat{X}_{\Phi_n}(t), v) &=& (X_0,v) + \int_0^t \langle A(s, \hat{X}_{\Phi_n}(s)), v\rangle ds  + \int_0^t (\Phi_n(s, X_{\Phi}(s)), v\rangle ds  \nonumber\\
  && + \int_0^t (B(s, X_{\Phi}(s)), v)\d W(s), \  \text{$\P$-a.s.}
\end{eqnarray}
for all $v\in \V$, $t\in [0,T]$ and for $n\in \N$.

\begin{theorem}\label{2312-22th4.1}
Assume that the embedding $\V\subset\H$ is compact and Hypothesis \ref{2312-22hypo1}  $A1$, $B1$, and $C1$ hold. Then, for any initial data $x\in \H$, there exists a unique probabilistic strong solution to the system \eqref{2312-25eq4.2}.
Furthermore, for any $p\geq 2$, the following estimate holds:
\begin{eqnarray}\label{2312-25eq4.3}
&& \E\Bigg[  \sup_{t\in[0,T]} \|\hat{X}_{\Phi_n}(t)\|_\H^p + \bigg( \int_0^T  \|\hat{X}_{\Phi_n}(t)\|_\V^\beta  \d t\bigg)^{\frac p2} \Bigg]   \nonumber\\
&\leq& C \bigg(  \| x\|_\H^p + \E\Big[\sup_{t\in[0,T]} \|X_{\Phi}(t)\|_\H^{p} \Big]\bigg).
\end{eqnarray}
\end{theorem}

Now, we consider  the   stochastic system  \eqref{2312-25eq4.2} in the finite-dimensional  space $\H_m$, for any $m\geq 1$
\begin{eqnarray}\label{2312-25eq4.4}
 (\hat{X}_{m, \Phi_n}(t), v) &=& (\PP_m x ,v) + \int_0^t \langle  \PP_m A(s, \hat{X}_{m, \Phi_n}(s)), v\rangle ds  + \int_0^t (\PP_m \Phi_n(s, X_{\Phi}(s)), v\rangle ds  \nonumber\\
  && + \int_0^t (\PP_m B(s, X_{\Phi}(s)), v) \Q_m \d W(s), \  \text{$\P$-a.s.}
\end{eqnarray}

In this  stochastic system  \eqref{2312-25eq4.4} above, the control term and fluctuation term only dependent on the system $X_{\Phi}$ correspond to the optimal control $\Phi$, not dependent on $\hat{X}_{m, \Phi_n}$.
For this  finite-dimensional stochastic system  \eqref{2312-25eq4.4}, we have  the following uniform energy estimate for $\hat{X}_{m, \Phi_n}$.

\begin{lemma}
  For any $p\geq 2$, there exists a constant $C>0$ such that uniform energy estimate holds:
\begin{eqnarray}\label{2312-25eq4.5}
&&\sup_{m\in \N} \E\Bigg[  \sup_{t\in[0,T]} \|\hat{X}_{m, \Phi_n}(t)\|_\H^p + \bigg( \int_0^T  \|\hat{X}_{m,\Phi_n}(t)\|_\V^\beta  \d t\bigg)^{\frac p2} \Bigg]   \nonumber\\
&\leq& C \bigg(  \| x\|_\H^p + \E\Big[\sup_{t\in[0,T]} \|X_{\Phi}(t)\|_\H^{p} \Big]\bigg).
\end{eqnarray}
\end{lemma}

\begin{proof}
Applying It\^{o}'s formula to the process $\| \hat{X}_{m, \Phi_n}(t) \|_\H^p$, and then using Young's inequality, we find
\begin{eqnarray}\label{2312-25eq4.11}
&& \| \hat{X}_{m, \Phi_n}(t) \|_\H^p   \nonumber \\
 &=& \|\PP_m x\|_\H^p + p\int_0^t \| \hat{X}_{m, \Phi_n}(s) \|_\H^{p-2} \langle A(s, \hat{X}_{m, \Phi_n}(s)), \hat{X}_{m, \Phi_n}(s) \rangle  \d s  \nonumber \\
 && + \frac{p}{2} \int_0^t \| \hat{X}_{m, \Phi_n}(s) \|_\H^{p-2} \big[ 2(\Phi_n(s, X_{ \Phi}(s)), \hat{X}_{m, \Phi_n}(s))     \big]  \d s  \nonumber \\
 && + \frac{p}{2} \int_0^t \| \hat{X}_{m, \Phi_n}(s) \|_\H^{p-2}     \|\PP_m B(s, X_{\Phi}(s)) Q_m\|_{L_2}^2     \d s  \nonumber \\
 && + \frac{p(p-2)}{2} \int_0^t \| \hat{X}_{m, \Phi_n}(s) \|_\H^{p-4}  \|\hat{X}_{m, \Phi_n}(s)   \circ \PP_m B(s, X_{\Phi}(s)) Q_m \|_\U^2     \d s  \nonumber \\
 && +  p \int_0^t \| \hat{X}_{m, \Phi_n}(s) \|_\H^{p-2} \Big(  B(s, X_{\Phi}(s)) Q_m \d W(s),   \hat{X}_{m, \Phi_n}(s)\Big)  \nonumber \\
 &\leq& \|\PP_m x\|_\H^p +  p\int_0^t \| \hat{X}_{m, \Phi_n}(s) \|_\H^{p-2} \langle A(s, \hat{X}_{m, \Phi_n}(s)), \hat{X}_{m, \Phi_n}(s) \rangle  \d s  \nonumber \\
 &&  +   \frac{p}{2} \int_0^t \Big[C_{\epsilon,p} \| \hat{X}_{m, \Phi_n}(s) \|_\H^{p}  + C_{\epsilon,p} \|\Phi_n(s, X_{ \Phi}(s))\|_\H^{p}  \Big]   \d s  \nonumber \\
 && +  \frac{p(p-1)}{2} \int_0^t \Big[C_{\epsilon,p} \| \hat{X}_{m, \Phi_n}(s) \|_\H^{p}  + C_{\epsilon,p}   \|\PP_m B(s, X_{\Phi}(s)) Q_m\|_{L_2}^{p}  \Big]     \d s  \nonumber \\
 && +  p \int_0^t \| \hat{X}_{m, \Phi_n}(s) \|_\H^{p-2} \Big(  B(s, X_{\Phi}(s)) Q_m \d W(s),   \hat{X}_{m, \Phi_n}(s)\Big).
\end{eqnarray}
Using Hypothesis \ref{2312-22hypo1} (A.3), Hypothesis \ref{2312-22hypo2} (B2) and Hypothesis \ref{2312-22hypo3} (C3), we obtain
\begin{eqnarray}\label{2312-25eq4.12}
 && \| \hat{X}_{m, \Phi_n}(t) \|_\H^p  + \frac{pC}{2} \int_0^t \| \hat{X}_{m, \Phi_n}(s) \|_\H^{p-2}  \| \hat{X}_{m, \Phi_n}(s) \|_\V^{\beta} \d s  \nonumber \\
 &\leq& \|\PP_m x\|_\H^p + \frac{p}{2} \int_0^t f_A(s) (1+\|\hat{X}_{m, \Phi_n}(s)\|_\H^2) \| \hat{X}_{m, \Phi_n}(s) \|_\H^{p-2}   \d s  \nonumber \\
 &&  +  \frac{p}{2} \int_0^t \Big[C_{\epsilon,p} \| \hat{X}_{m, \Phi_n}(s) \|_\H^{p}  + C_{\epsilon,p}(f_\Phi(s)+\| X_{ \Phi}(s)\|_\H^{p})    \Big]   \d s  \nonumber \\
 && +  \frac{p(p-1)}{2} \int_0^t \Big[C_{\epsilon,p} \| \hat{X}_{m, \Phi_n}(s) \|_\H^{p}  + C_{\epsilon,p}  g_B(s)(1+\| X_{ \Phi}(s)\|_\H^{p})   \Big] \d s  \nonumber \\
 && +  p \int_0^t \| \hat{X}_{m, \Phi_n}(s) \|_\H^{p-2} \Big(  B(s, X_{\Phi}(s)) Q_m \d W(s),   \hat{X}_{m, \Phi_n}(s)\Big)  \nonumber \\
  &\leq& \|\PP_m x\|_\H^p + C_p \int_0^t \Big[ f_A(s)  +  f_A(s)\| \hat{X}_{m, \Phi_n}(s) \|_\H^{p} \Big]  \d s  \nonumber \\
 &&  +  \frac{p}{2} \int_0^t \Big[C_{\epsilon,p} \| \hat{X}_{m, \Phi_n}(s) \|_\H^{p}  + C_{\epsilon,p}(f_\Phi(s)+\| X_{ \Phi}(s)\|_\H^{p})    \Big]   \d s  \nonumber \\
 && +  \frac{p(p-1)}{2} \int_0^t \Big[C_{\epsilon,p} \| \hat{X}_{m, \Phi_n}(s) \|_\H^{p}  + C_{\epsilon,p}  g_B(s)(1+\| X_{ \Phi}(s)\|_\H^{p})   \Big] \d s  \nonumber \\
 && +  p \int_0^t \| \hat{X}_{m, \Phi_n}(s) \|_\H^{p-2} \Big(  B(s, X_{\Phi}(s)) Q_m \d W(s),   \hat{X}_{m, \Phi_n}(s)\Big)  \nonumber \\
 &=& \|\PP_m x\|_\H^p +  C_{\epsilon,p} \int_0^t \Big[ F(s)+\| X_{ \Phi}(s)\|_\H^{p}+  g_B(s)\| X_{ \Phi}(s)\|_\H^{p})   \Big]  \d s  \nonumber \\
 && +  C_{\epsilon,p}  \int_0^t \Big[ \Big(   f_A(s) +1   \Big) \| \hat{X}_{m, \Phi_n}(s) \|_\H^{p} \Big]   \d s \nonumber \\
 && +  p \int_0^t \| \hat{X}_{m, \Phi_n}(s) \|_\H^{p-2} \Big(  B(s, X_{\Phi}(s)) Q_m \d W(s),   \hat{X}_{m, \Phi_n}(s)\Big),
\end{eqnarray}
where $F(s)=f_A(s)  +f_\Phi(s)+  g_B(s)$ as \eqref{2312-25eq3.3020}.
Define sequence of stopping times as follows:
\begin{align*}
	\tau_N^m:=T\wedge \inf\{t\geq 0:\|\hat{X}_{m, \Phi_n}(t)\|_\H>N\}.
\end{align*}
Then $\tau_N^m\to T,\;\P$-a.s., as $N\to \infty$ for every $m$. Next, taking the supremum over time from $0$ to $r\wedge\tau_N^m$ and then taking  expectations on both sides of the above inequality \eqref{2312-25eq4.12}, we deduce
\begin{eqnarray}\label{2312-25eq4.13}
 && \E\bigg[ \sup_{t\in[0,r\wedge\tau_N^m]} \| \hat{X}_{m, \Phi_n}(t) \|_\H^p \bigg] + \frac{pC}{2} \E\bigg[ \int_0^{r\wedge\tau_N^m} \| \hat{X}_{m, \Phi_n}(s) \|_\H^{p-2}  \| \hat{X}_{m, \Phi_n}(s) \|_\V^{\beta} \d s  \bigg] \nonumber \\
 &\leq&   \| x\|_\H^p + C_{\epsilon,p}  \E\bigg[ \int_0^{r\wedge\tau_N^m} \Big[F(s)+\| X_{ \Phi}(s)\|_\H^{p}+  g_B(s)\| X_{ \Phi}(s)\|_\H^{p})   \Big]  \d s  \bigg] \nonumber \\
 && +  C_{\epsilon,p}  \E\bigg[  \int_0^{r\wedge\tau_N^m} \Big[ \big(   f_A(s) +1   \big) \| \hat{X}_{m, \Phi_n}(s) \|_\H^{p} \Big]   \d s \bigg]\nonumber \\
 && +  p  \E\bigg[  \sup_{t\in[0,r\wedge\tau_N^m]} \int_0^t \| \hat{X}_{m, \Phi_n}(s) \|_\H^{p-2} \Big(  B(s, X_{\Phi}(s)) Q_m \d W(s),   \hat{X}_{m, \Phi_n}(s)\Big) \bigg] .
\end{eqnarray}
Now, we  consider the last term in \eqref{2312-25eq4.13} and estimate it using Hypothesis \ref{2312-22hypo1} (B3), Burkholder-Davis-Gundy inequality, Young's and Holder's inequalities as
\begin{eqnarray}\label{2312-25eq4.14}
 &&    p  \E\bigg[  \sup_{t\in[0,r\wedge\tau_N^m]} \int_0^t \| \hat{X}_{m, \Phi_n}(s) \|_\H^{p-2} \Big(  B(s, X_{\Phi}(s)) Q_m \d W(s),   \hat{X}_{m, \Phi_n}(s)\Big) \bigg]  \nonumber \\
 &\leq&   C_p \E\bigg[ \bigg( \int_0^{r\wedge\tau_N^m}  \| \hat{X}_{m, \Phi_n}(s) \|_\H^{2p-2}  \|B(s, X_{\Phi}(s))\|_{L_2}^2 \d s \bigg)^{\frac12} \bigg]  \nonumber \\
 &\leq&  C_p  \E\bigg[  \bigg( \sup_{s\in[0, r\wedge\tau_N^m]} \| \hat{X}_{m, \Phi_n}(s) \|_\H^{2p-2} \int_0^{r\wedge\tau_N^m}  \|B(s, X_{\Phi}(s))\|_{L_2}^2 \d s \bigg)^{\frac12} \bigg]  \nonumber \\
  &\leq& \epsilon \E\bigg[    \sup_{s\in[0, r\wedge\tau_N^m]} \| \hat{X}_{m, \Phi_n}(s) \|_\H^{p} \bigg] +   C_{\epsilon, p}  \E\bigg[   \bigg(  \int_0^{r\wedge\tau_N^m}  \|B(s, X_{\Phi}(s))\|_{L_2}^2 \d s \bigg)^{\frac{p}{2}}  \bigg]  \nonumber \\
  &\leq&   \epsilon \E\bigg[    \sup_{s\in[0, r\wedge\tau_N^m]} \| \hat{X}_{m, \Phi_n}(s) \|_\H^{p} \bigg] +   C_{\epsilon, p}  \E\bigg[   \int_0^{r\wedge\tau_N^m}  \|B(s, X_{\Phi}(s))\|_{L_2}^p  \d s   \bigg]  \nonumber \\
   &\leq&   \epsilon \E\bigg[    \sup_{s\in[0, r\wedge\tau_N^m]} \| \hat{X}_{m, \Phi_n}(s) \|_\H^{p} \bigg] +   C_{\epsilon, p}  \E\bigg[  \int_0^{r\wedge\tau_N^m}  [ g_B(s)  +  g_B(s) \|X_{\Phi}(s)\|_\H^{p}  ]\d s   \bigg],  \nonumber \\
\end{eqnarray}
where $\epsilon>0$.
Substituting \eqref{2312-25eq4.14} in \eqref{2312-25eq4.13}, and choosing an appropriate parameter $\epsilon$, we get
\begin{eqnarray}\label{2312-25eq4.15}
 && \E\bigg[ \sup_{t\in[0,r\wedge\tau_N^m]} \| \hat{X}_{m, \Phi_n}(t) \|_\H^p \bigg] + C \E\bigg[ \int_0^{r\wedge\tau_N^m} \| \hat{X}_{m, \Phi_n}(s) \|_\H^{p-2}  \| \hat{X}_{m, \Phi_n}(s) \|_\V^{\beta} \d s  \bigg] \nonumber \\
 &\leq&    \| x\|_\H^p + C_{\epsilon,p}  \E\bigg[ \int_0^{r\wedge\tau_N^m} \Big[ F(s)+\| X_{ \Phi}(s)\|_\H^{p}+  g_B(s)\| X_{ \Phi}(s)\|_\H^{p})   \Big]  \d s  \bigg] \nonumber \\
 && +  C_{\epsilon,p}  \E\bigg[  \int_0^{r\wedge\tau_N^m} \Big[ \big(   f_A(s) +1   \big) \| \hat{X}_{m, \Phi_n}(s) \|_\H^{p} \Big]   \d s \bigg]\nonumber \\
 &&   +   C_{\epsilon, p}  \E\bigg[  \int_0^{r\wedge\tau_N^m}  [ g_B(s)   + g_B(s)\|X_{\Phi}(s)\|_\H^{p}  ]\d s   \bigg]  \nonumber \\
 &\leq&   \| x\|_\H^p + C_{\epsilon, p}  \E\bigg[ \int_0^{r\wedge\tau_N^m} \Big[  F(s)   +  g_B(s) \|X_{\Phi}(s)\|_\H^{p}   \Big]  \d s  \bigg] \nonumber \\
 && +   C_{\epsilon,p}  \E\bigg[  \int_0^{r\wedge\tau_N^m} \Big[ \big( 1+  f_A(s)   \big) \| \hat{X}_{m, \Phi_n}(s) \|_\H^{p}\Big]   \d s \bigg].
\end{eqnarray}
Using Gronwall's inequality, passing $N\to \infty$ and applying Fatou's lemma, we get
\begin{eqnarray}\label{2312-25eq4.16}
 && \E\bigg[ \sup_{t\in[0,T]} \| \hat{X}_{m, \Phi_n}(t) \|_\H^p \bigg]   \nonumber \\
 &\leq&   C_{\epsilon,p}   \bigg(  \| x\|_\H^p + C_{\epsilon, p}  \E\bigg[ \int_0^{T} \Big[  F(s)   +  g_B(s) \|X_{\Phi}(s)\|_\H^{p}   \Big]  \d s  \bigg] \bigg)  \nonumber \\
 &&  ~~~~ \cdot \exp\bigg( C\int_0^T \big( 1+  f_A(s)   \big) \d s \bigg)   \nonumber \\
 &\leq&   C_{\epsilon,p}   \bigg(  \| x\|_\H^p + \E\Big[\sup_{t\in[0,T]} \|X_{\Phi}(t)\|_\H^{p}\Big] \bigg)   .
\end{eqnarray}

Applying It\^{o}'s formula to the process $\| \hat{X}_{m, \Phi_n}(t) \|_\H^2$, and applying Hypothesis \ref{2312-22hypo1},  we find
\begin{eqnarray}\label{2312-25eq4.17}
 && \| \hat{X}_{m, \Phi_n}(t) \|_\H^2  \nonumber \\
 &=& \|\PP_m x\|_\H^2 + 2\int_0^t   \langle A(s, \hat{X}_{m, \Phi_n}(s)), \hat{X}_{m, \Phi_n}(s) \rangle  \d s  \nonumber \\
 && +   \int_0^t   \big[ 2(\Phi_n(s, X_{ \Phi}(s)), \hat{X}_{m, \Phi_n}(s))     \big]  \d s
    +   \int_0^t   \|\PP_m B(s, X_{\Phi}(s)) Q_m\|_{L_2}^2     \d s  \nonumber \\
 && +  2 \int_0^t   \Big(  B(s, X_{\Phi}(s)) Q_m \d W(s),   \hat{X}_{m, \Phi_n}(s)\Big) \nonumber \\
 &\leq& \|\PP_m x\|_\H^2 +  \int_0^t \Big[ f_A(t)(1+\|\hat{X}_{m, \Phi_n}(s)\|_\H^2) - C\|\hat{X}_{m, \Phi_n}(s)\|_\V^\beta \Big]  \d s  \nonumber \\
 && +   \int_0^t   \big[ 2(\Phi_n(s, X_{ \Phi}(s)), \hat{X}_{m, \Phi_n}(s))     \big]  \d s
    +   \int_0^t   \|\PP_m B(s, X_{\Phi}(s)) Q_m\|_{L_2}^2     \d s  \nonumber \\
 && +  2 \int_0^t   \Big(  B(s, X_{\Phi}(s)) Q_m \d W(s),   \hat{X}_{m, \Phi_n}(s)\Big)   .
\end{eqnarray}
Hence, tanking the expectation on \eqref{2312-25eq4.17}, it follows that
\begin{eqnarray}\label{2312-25eq4.18}
 && \E\bigg[ \| \hat{X}_{m, \Phi_n}(t) \|_\H^2 + C\int_0^t  \|\hat{X}_{m, \Phi_n}(s)\|_\V^\beta   \d s \bigg]   \nonumber \\
 &\leq& \E\bigg[\|\PP_m x\|_\H^2 +  \int_0^t \Big[ f_A(t)(1+\|\hat{X}_{m, \Phi_n}(s)\|_\H^2) \Big] \d s  \nonumber \\
 && ~~~~~ +   \int_0^t   \big[ 2(\Phi_n(s, X_{ \Phi}(s)), \hat{X}_{m, \Phi_n}(s))     \big]  \d s \nonumber \\
 &&  ~~~~~  +   \int_0^t   \|\PP_m B(s, X_{\Phi}(s)) Q_m\|_{L_2}^2     \d s  \bigg]    .
\end{eqnarray}
Using the same argument above for \eqref{2312-25eq4.18}, we obtain
\begin{eqnarray}\label{2312-25eq4.19}
  \E\bigg[  \int_0^t  \|\hat{X}_{m, \Phi_n}(s)\|_\V^\beta   \d s \bigg]
 \leq   C    \bigg(  \| x\|_\H^p + \E\Big[\sup_{t\in[0,T]} \|X_{\Phi}(t)\|_\H^{p}\Big] \bigg)    .
\end{eqnarray}
Combining \eqref{2312-25eq4.16} and \eqref{2312-25eq4.19}, we get our desired \eqref{2312-25eq4.5}.
\end{proof}

Same arguments as in section above, we can obtain Theorem \ref{2312-22th4.1}.
Now, we establish some convergence properties of the solutions to the systems above.

\begin{theorem}\label{2312-22th4.2}
Consider $(\Phi_n)_{n\in\N} \subset \mathcal{U}$ and $\Phi\in\mathcal{U}$ such that
\begin{equation}\label{2312-25eq4.20.1}
  \lim_{n\to\infty} \|\Phi_n(t,x) -\Phi(t,x)\|_\H =0,
\end{equation}
for all $t\in[0,T], x\in\V$.
Then the solutions to  (\ref{2312-22eq01.1})  and \eqref{2312-25eq4.2} satisfy the following convergence
\begin{eqnarray}\label{2312-25eq4.20}
 && \lim_{n\to\infty} \E\bigg[ \sup_{t\in[0,T]}   \| \hat{X}_{\Phi_n}(t)  - X_{\Phi}(t)  \|_\V^2 \bigg] =0.
\end{eqnarray}
\end{theorem}

\begin{proof}
  Consider the following equation without control term
\begin{equation}\label{2312-25eq4.21}
   (z(t), v) = (X_0, v) + \int_0^t \langle A(s, z(s)), v\rangle \d s  +  \int_0^t(B(s,X_\Phi(s)), v) \d W(s), 
\end{equation}
for all $v\in \V$, $t\in[0,T]$.
Similar argument as in \eqref{2312-25eq4.2}, there exists a unique solution $z\in C([0,T]; \H) \cap L^2([0,T]; \V)$ of \eqref{2312-25eq4.21}, furthermore,   the following energy estimates holds:
 \begin{eqnarray}\label{2312-25eq4.22}
&&  \E\Bigg[  \sup_{t\in[0,T]} \|z(t)\|_\H^2 + \int_0^T  \|z(t)\|_\V^\beta  \d t  \Bigg]   \nonumber\\
&\leq&  C(\|x\|_\H^2 + \E \Big[ \sup_{t\in[0,T]} \|X_\Phi(t)\|_\H^2  \Big].
\end{eqnarray}
Hence, there is a $\kappa_1(\omega)>0$ such that
 \begin{eqnarray}\label{2312-25eq4.23}
\sup_{t\in[0,T]} \|z(t)\|_\H^2 \leq \kappa_1(\omega),
&& \int_0^T  \|z(t)\|_\V^\beta  \d t \leq \kappa_1(\omega),
\end{eqnarray}
and
\begin{eqnarray}\label{2312-25eq4.24}
\sup_{t\in[0,T]} \|X_\Phi(t)\|_\H^2 \leq \kappa_1(\omega),
&& \int_0^T  \|X_\Phi(t)\|_\V^\beta  \d t \leq \kappa_1(\omega).
\end{eqnarray}
Due to
\begin{eqnarray}\label{2312-25eq4.25}
&& \E\bigg[ \bigg\|\int_s^t B(r, X_\Phi(r)) \d W(r) \bigg\|_\H^4 \bigg] \leq C(t-s)^2,
\end{eqnarray}
and then applying the Kolmogorov continuity theorem, there are random variable $\kappa_2(\omega)$ and $\gamma\in (0,\frac12)$ such that
\begin{eqnarray}\label{2312-25eq4.26}
&&   \bigg\|\int_s^t B(r, X_\Phi(r)) \d W(r) \bigg\|_\H^2   \leq \kappa_2(\omega)|t-s|^{\gamma}, ~~ \P\text{-a.s.}
\end{eqnarray}
for any $s,t\in [0,T] $.

Let $\tilde{\Omega} \subset \Omega$ with $\P(\tilde\Omega) =1$ such that for any $\omega\in \tilde\Omega$ the equations in  (\ref{2312-22eq01.1})  and \eqref{2312-25eq4.21} are satisfied, and for each $n\in \N$, \eqref{2312-25eq4.2} is satisfied and the inequalities in \eqref{2312-25eq4.23}, \eqref{2312-25eq4.24} and \eqref{2312-25eq4.26} are also satisfied.

It follows from \eqref{2312-25eq4.4}, \eqref{2312-25eq4.21}, \eqref{2312-25eq4.24} and the assumption on coefficients that for $\omega\in \tilde\Omega$, there is a $\kappa_3(\omega)$  independent of $n$ such that
\begin{eqnarray}\label{2312-25eq4.27}
 \sup_{t\in[0,T]} \| \hat{X}_{\Phi_n}(t) - z(t)\|^2 + C\int_0^T \|\hat{X}_{\Phi_n}(s) - z(s)\|_\V^\beta \d s  \leq \kappa_3(\omega).
\end{eqnarray}
Combining \eqref{2312-25eq4.27} and \eqref{2312-25eq4.22}, for all $n\in \N$, we deduce that for  $\omega\in \tilde\Omega$, there is a $\kappa_4(\omega)$  independent of $n$ such that
\begin{eqnarray}\label{2312-25eq4.28}
 \sup_{t\in[0,T]} \| \hat{X}_{\Phi_n}(t)\|^2 + C\int_0^T \|\hat{X}_{\Phi_n}(s)\|_\V^\beta \d s  \leq \kappa_4(\omega).
\end{eqnarray}
For  $\omega\in \tilde\Omega$, consider the sequence $G(\omega):= \{\hat{X}_{\Phi_n}(\omega, \cdot)\}_{n=1}^\infty$ which is bounded by the aid of \eqref{2312-25eq4.28}.

It follows from \eqref{2312-25eq4.2} that for any $s,t\in[0,T]$ and $s<t$,
\begin{eqnarray}\label{2312-25eq4.29}
 && \|\hat{X}_{\Phi_n}(t) - \hat{X}_{\Phi_n}(s)\|_\H^2  \nonumber\\
&\leq& (t-s) \int_s^t \|A(r, \hat{X}_{\Phi_n}(r))\|_\H^2 \d r     \nonumber\\
&&  + (t-s) \int_s^t \|\Phi_n(r, \hat{X}_{\Phi}(r))\|_\H^2 \d r  +    \Big \|\int_s^t  B(r, \hat{X}_{\Phi}(r)) \d W(r) \Big\|_\H^2.
\end{eqnarray}
By \eqref{2312-25eq4.26}, we can deduce that
\begin{eqnarray}\label{2312-25eq4.30}
  \|\hat{X}_{\Phi_n}(t) - \hat{X}_{\Phi_n}(s)\|_\H^2
&\leq& \kappa(\omega)(t-s) + \kappa_2(\omega)(t-s)^{2\gamma}
\end{eqnarray}
 for  $\gamma\in (0, \frac12)$ and where $\kappa(\omega)>0$ is independent of $n$.
 Consequently, $G(\omega)$ is equi-continuous in $C(0,T; \H)$.
 Applying Dubinsky's theorem, we get that $G(\omega)$ is relatively compact in $L^2(0,T; \H)$. Thus there exists a subsequence $n_k$ of $n$ and $\hat{X}\in L^2(0,T;\H)$ such that
 \begin{eqnarray}\label{2312-25eq4.31}
 \lim_{k\to\infty} \int_0^T \|\hat{X}_{\Phi_{n_k}}(s) - \hat{X}(s)\|_\H^2 \d s =0.
\end{eqnarray}
Using the assumption on operator $A$ we get
 \begin{eqnarray*}\label{2312-25eq4.32}
 2\langle A(s, \hat{X}_{\Phi_n}) -  A(s, \hat{X}_{\Phi}), \hat{X}_{\Phi_n} - \hat{X}_{\Phi}\rangle
 \leq [f_A(t)+\rho(\hat{X}_{\Phi_n}) + \eta(\hat{X}_{\Phi})]    \| \hat{X}_{\Phi_n} - \hat{X}_{\Phi} \|_\H^2,
\end{eqnarray*}
and
\begin{eqnarray*}\label{2312-25eq4.33}
 2\langle A(s, \hat{X}_{\Phi}) -  A(s, \hat{X}_{\Phi}), \hat{X}_{\Phi} - \hat{X}_{\Phi}\rangle
 \leq [f_A(t)+\rho(\hat{X}_{\Phi}) + \eta(\hat{X}_{\Phi})]    \| \hat{X}_{\Phi} - \hat{X}_{\Phi} \|_\H^2,
\end{eqnarray*}
hence, by virtue of $\varphi(t) = C e^{-\int_0^t f_A(s)+\rho(s) + \eta(s) \d s}$ and Poincar\'{e}'s inequality,  we obtain
\begin{eqnarray}\label{2312-25eq4.34}
 && \varphi(t) \| \hat{X}_{\Phi_n}(t)  - \hat{X}_{\Phi}(t)  \|_\V^2     \nonumber\\
 &=& 2\int_0^t \varphi(s) \langle A(s, \hat{X}_{\Phi_n}(s) ) -  A(s, \hat{X}_{\Phi}(s) ), \hat{X}_{\Phi_n}(s)  - \hat{X}_{\Phi}(s) \rangle \d s     \nonumber\\
 && + 2\int_0^t \varphi(s)  (\Phi_n(s, X_{\Phi}(s) ) -  \Phi(s, X_{\Phi}(s) ), \hat{X}_{\Phi_n}(s)  - \hat{X}_{\Phi})(s)    \d s     \nonumber\\
 && - \varphi(s) [f_A(t)+\rho(\hat{X}_{\Phi}(s)) + \eta(\hat{X}_{\Phi}(s))]    \| \hat{X}_{\Phi}(s) - \hat{X}_{\Phi}(s) \|_\V^2   \nonumber\\
 &\leq&  2\int_0^t \varphi(s)  (\Phi_n(s, X_{\Phi}(s)) -  \Phi(s, X_{\Phi}(s)), \hat{X}_{\Phi_n}(s) - \hat{X}_{\Phi}(s))   \d s,
\end{eqnarray}
and
\begin{eqnarray}\label{2312-25eq4.35}
 && \varphi(t) \| \hat{X}_{\Phi}(t)  - X_{\Phi}(t)  \|_\V^2     \nonumber\\
 &=&  2\int_0^t \varphi(s)  (A(s, \hat{X}_{\Phi}(s)) -  A(s, X_{\Phi}(s)), \hat{X}_{\Phi}(s) - X_{\Phi}(s))   \d s   \nonumber\\
 && - \varphi(s) [f_A(t)+\rho(\hat{X}_{\Phi}(s)) + \eta(X_{\Phi}(s))]    \| \hat{X}_{\Phi}(s) - X_{\Phi}(s) \|_\H^2 \nonumber\\
 &\leq& 0.
\end{eqnarray}
Hence, taking the supremum over $[0,T]$ on both sides of \eqref{2312-25eq4.34}, and applying H\"{o}lder's inequality, we can obtain
 \begin{eqnarray}\label{2312-25eq4.37}
 \lim_{n\to\infty}\sup_{t\in[0,T]} \| \hat{X}_{\Phi_n}(t)  - \hat{X}_{\Phi}(t)  \|_\V^2  =0
\end{eqnarray}
because of \eqref{2312-25eq4.31} and   $ \lim_{n\to\infty} \|\Phi_n(t,x) -\Phi(t,x)\|_\H =0$.
Now, using the triangle inequality, we get
 \begin{eqnarray}\label{2312-25eq4.38}
  && \sup_{t\in[0,T]}  \| \hat{X}_{\Phi_n}(t)  - X_{\Phi}(t)  \|_\V^2  \nonumber\\
 &\leq& \sup_{t\in[0,T]} \| \hat{X}_{\Phi_n}(t)  - \hat{X}_{\Phi}(t)  \|_\V^2    +  \sup_{t\in[0,T]} \| \hat{X}_{\Phi}(t)  - X_{\Phi}(t)  \|_\V^2,
\end{eqnarray}
combining \eqref{2312-25eq4.35}, \eqref{2312-25eq4.37} and \eqref{2312-25eq4.38}, we obtain
\begin{eqnarray}\label{2312-25eq4.39}
 \lim_{n\to\infty}  \sup_{t\in[0,T]}   \| \hat{X}_{\Phi_n}(t)  - X_{\Phi}(t)  \|_\V^2  =0.
\end{eqnarray}
Since every subsequence of $\{\hat{X}_{\Phi_n}(\omega,\cdot)\}$ has a subsequence which converges to the same limit $X_\Phi(\omega,\cdot)$ in the space $L^2(0,T; \V)$, it follows that the sequence $\{\hat{X}_{\Phi_n}(\omega,\cdot)\}$ converges to $X_\Phi(\omega,\cdot)$.
It follows from Theorem \ref{2312-22th4.1}, the stochastic process $\{\hat{X}_{\Phi_n}(t)\}_{t\in[0,T]}$ and $\{X_{\Phi}(t)\}_{t\in[0,T]}$ are uniformly integrable and thus the desired result follows.
\end{proof}

For fixed $M>0$, we define the stopping time
\begin{eqnarray}\label{2312-25eq4.40}
  \tau_m :=
  \left\{
    \begin{array}{ll}
      T, & \hbox{if $ \sup_{t\in [0,T]} \|X_{\Phi}(t)\|_\V^2 < M$;} \\
     \inf\Big\{  t\in[0,T]: \|X_{\Phi}(t)\|_\V^2 \geq M \Big\}, & \hbox{otherwise.}
    \end{array}
  \right.
\end{eqnarray}

Let $\Phi_n$ and $\Phi$ be the sequence and the map obtained in the Lemma \ref{2212-22lem4.1}, the following theorem asserts that there is a subsequence $n_k$ of $n$ such that the corresponding solutions $X_{\Phi_{n_k}}$ of  (\ref{2312-22eq01.1})  converge strongly to $X_{\Phi}$.

\begin{theorem}\label{2312-22th4.3}
Consider $(\Phi_n)_{n\in\N} \subset \mathcal{U}$ and $\Phi\in\mathcal{U}$ such that
\begin{equation*}\label{2312-25eq4.41}
  \lim_{n\to\infty} \|\Phi_n(t,x) -\Phi(t,x)\|_\H =0,
\end{equation*}
for all $t\in[0,T], x\in\V$.
Then there is a subsequence $n_k$ of $n$ such that the solutions to  (\ref{2312-22eq01.1})  and \eqref{2312-25eq4.2} satisfy the following convergence
\begin{eqnarray}\label{2312-25eq4.42}
 && \lim_{n\to\infty} \E\bigg[ \sup_{t\in[0,T]}   \| X_{\Phi_{n_k}}(t)  - X_{\Phi}(t)  \|_\V^2 \bigg] =0.
\end{eqnarray}
\end{theorem}

\begin{proof}
Applying the It\^{o}'s formula, we get
\begin{eqnarray}\label{2312-25eq4.43}
 && \phi(\tau_M) \|\hat{X}_{\Phi_n}(\tau_M) - X_{\Phi_n}(\tau_M) \|_\V^2    \nonumber\\
 &=&  \int_0^{\tau_m}  \phi'(s) \| \hat{X}_{\Phi_n}(s) - X_{\Phi_n}(s) \|_\V^2 \d s  \nonumber\\
&&  + 2 \int_0^{\tau_M}  \phi(s) (\Phi_n(s, X_\Phi(s)) - \Phi_n(s, X_{\Phi_n}(s)),  \hat{X}_{\Phi_n}(s) - X_{\Phi_n}(s) ) \d s   \nonumber\\
&&  + 2 \int_0^{\tau_M}  \phi(s) \langle A(s, \hat{X}_{\Phi_n}(s)) - A(s, X_{\Phi_n}(s)),  \hat{X}_{\Phi_n}(s) - X_{\Phi_n}(s) \rangle \d s    \nonumber\\
&&  + 2 \int_0^{\tau_M}  \phi(s) (B(s, X_{\Phi}(s)) - B(s, X_{\Phi_n}(s)),  \hat{X}_{\Phi_n}(s) - X_{\Phi_n}(s) ) \d s     \nonumber\\
&&  + \int_0^{\tau_M}  \phi(s) \| B(s,X(s)) - B(s,X_{\Phi_n}(s)) \|_\V^2 \d s  .
\end{eqnarray}
Using the assumption on $A, \Phi_n, C$, we get
\begin{eqnarray}\label{2312-25eq4.44}
 && \phi(\tau_M) \|\hat{X}_{\Phi_n}(\tau_M) - X_{\Phi_n}(\tau_M) \|_\V^2    \nonumber\\
 &\leq&  \int_0^{\tau_m}  \phi'(s) \| \hat{X}_{\Phi_n}(s) - X_{\Phi_n}(s) \|_\V^2 \d s  \nonumber\\
&&  +   \int_0^{\tau_M}  \phi(s) [\alpha\|X_{\Phi_n}(s)-X_\Phi(s)\|_\H^2 +  \| \hat{X}_{\Phi_n}(s) - X_{\Phi_n}(s) \|_\H^2] \d s   \nonumber\\
&&  +   \int_0^{\tau_M}  \phi(s)  [f_A(s) + \rho(\hat{X}_{\Phi_n}(s))  + \eta( X_{\Phi_n}(s)) ]  \|  \hat{X}_{\Phi_n}(s) - X_{\Phi_n}(s) \|_\H^2 \d s    \nonumber\\
&&  +   \int_0^{\tau_M}  \phi(s) [\|B(s, X_{\Phi}(s)) - B(s, X_{\Phi_n}(s))\|_{L_2}^2  +  \|\hat{X}_{\Phi_n}(s) - X_{\Phi_n}(s) \|_\H^2 ] \d s     \nonumber\\
&&  + \int_0^{\tau_M}  \phi(s) \|B(s,X(s)) - B(s,X_{\Phi_n}(s)) \|_{L_2}^2 \d s   \nonumber\\
 &\leq&  \int_0^{\tau_m} \Big[ \phi'(s)+ 2 \phi(s)    \nonumber\\
&& ~~~~~~~~~~~~ +  \phi(s)\big[f_A(s) + \rho(\hat{X}_{\Phi_n}(s))  + \eta( X_{\Phi_n}(s)) \big]  \Big]  \| \hat{X}_{\Phi_n}(s) - X_{\Phi_n}(s) \|_\H^2 \d s  \nonumber\\
&&  +  \alpha \int_0^{\tau_M}  \phi(s)\|X_{\Phi_n}(s)-X_\Phi(s)\|_\H^2   \d s      \nonumber\\
&&  +   2 \int_0^{\tau_M}  \phi(s) \|B(s,X(s)) - B(s,X_{\Phi_n}(s)) \|_{L_2}^2 \d s  .
\end{eqnarray}
Letting
$$\phi(t) := \exp\bigg(  - \int_0^t  \Big[2+ f_A(s) + \rho(\hat{X}_{\Phi_n}(s))  + \eta( X_{\Phi_n}(s)) \Big] \d s \bigg),$$
then substituting $\phi(s)$ into \eqref{2312-25eq4.44}, and then applying Poincar\'{e}'s inequality we obtain
\begin{eqnarray}\label{2312-25eq4.45}
 && \phi(\tau_M) \|\hat{X}_{\Phi_n}(\tau_M) - X_{\Phi_n}(\tau_M) \|_\V^2    \nonumber\\
 &\leq&   C \alpha \int_0^{\tau_M}  \phi(s)\|X_{\Phi_n}(s)-X_\Phi(s)\|_\V^2   \d s    \nonumber\\
&&    +   2 \int_0^{\tau_M}  \phi(s) \|B(s,X(s)) - B(s,X_{\Phi_n}(s)) \|_{L_2}^2 \d s    \nonumber\\
 &\leq&  ( 2C+ \alpha) \int_0^{\tau_M}  \phi(s)\|X_{\Phi_n}(s)-X_\Phi(s)\|_\V^2   \d s,
\end{eqnarray}
where in the last inequality we have used (B.1) in the Hypothesis \ref{2312-22hypo2}.

We can extract a subsequence $\{\hat{X}_{\Phi_{n_k}}\}_{k=1}^\infty$ from the sequence $\{\hat{X}_{\Phi_{n}}\}_{n=1}^\infty$ that converges to $X_\Phi$ for $(\omega,t)\in \Omega\times [0,T]$.
 By virtue of \eqref{2312-25eq4.45}, and the the triangle inequality, we obtain
\begin{eqnarray}\label{2312-25eq4.46}
&&  \phi(\tau_M) \|X_{\Phi_{n_k}}(\tau_M) - X_{\Phi}(\tau_M) \|_\V^2   \nonumber\\
 &\leq& \phi(\tau_M) \|X_{\Phi_{n_k}}(\tau_M) - \hat{X}_{\Phi_{n_k}}(\tau_M)   \|_\V^2    +  \phi(\tau_M) \|\hat{X}_{\Phi_{n_k}}(\tau_M) - X_{\Phi}(\tau_M) \|_\V^2    \nonumber\\
 &\leq& \phi(\tau_M) \|\hat{X}_{\Phi_{n_k}}(\tau_M) - X_{\Phi}(\tau_M) \|_\V^2    \nonumber\\
&&  + ( 2C+ \alpha) \int_0^{\tau_M}  \phi(s)\|X_{\Phi_{n_k}}(s)-X_\Phi(s)\|_\V^2   \d s,
\end{eqnarray}
Hence, taking the supremum over $[0,T]$, and applying Gronwall's inequality, and by Theorem \ref{2312-22th4.3}, we can obtain our desired result.
\end{proof}

\vspace{0.2cm}

Finally, we are in a position to formulate the main result of our present paper.

\begin{theorem}\label{2312-22th4.5}
 Assume that $f$ and $g$ satisfy the Hypothesis \ref{2312-22hypo14}, then there exists an optimal feedback control $\Phi^*\in \mathcal{U}$ for the control problem ($\mathcal{P}$).
\end{theorem}
\begin{proof}
  Let us consider a minimizing sequence $(\Phi_n)\subset \mathcal{U}$. Then we have
\begin{equation*}\label{2312-25eq4.47}
  \inf_{\Phi\in  \mathcal{U}} J(\Phi) = \lim_{n\to\infty} J(\Phi_n).
\end{equation*}
Due to Lemma \ref{2212-22lem4.1}, there exist a subsequence $(\Phi_{n_k})$ of $(\Phi_n)$ and an admissible control $\Phi^*$, such that for all $t\in[0,T]$ and $x, y\in \V$, we have
\begin{equation}\label{2312-25eq4.48}
 \lim_{k\to\infty} (\Phi_{n_k}(t,x) , y) = (\Phi^*(t,x), y).
\end{equation}
Due to Theorem \ref{2312-22th4.3}, there exists a subsequence of $(X_{\Phi_{n_k}})$ still denoted by $(X_{\Phi_{n_k}})$ such that for all $t\in[0,T]$,  and $\P$-a.s.
\begin{equation}\label{2312-25eq4.50}
 X_{\Phi_{n_k}}\to  X_{\Phi^*}, ~~  \text{ strong in } \V.
\end{equation}
The relations \eqref{2312-25eq4.48} and \eqref{2312-25eq4.50}, give that for all $t\in[0,T]$ and $\P$-a.s.
\begin{equation*}\label{2312-25eq4.51}
  \Phi_{n_k}(t, X_{\Phi_{n_k}}) \rightharpoonup \Phi^*(t, X_{\Phi^*}), ~~  \text{weakly in } \H.
\end{equation*}
Taking into account the lower semicontinuity assumptions in Hypothesis \ref{2312-22hypo14} and the Fatou's lemma, we deduce
\begin{eqnarray}\label{2312-25eq4.52}
 && \E \int_0^T \Big[f(t, X_{\Phi^*}) + g( \Phi^*(t, X_{\Phi^*})) \Big]\d t  \nonumber\\
&\leq& \E \int_0^T \liminf_{k\to\infty} \Big[f(t, X_{\Phi_{n_k}}) + g( \Phi_{n_k}(t, X_{\Phi_{n_k}}))\Big] \d t  \nonumber\\
&\leq& \liminf_{k\to\infty}  \E \int_0^T \Big[f(t,   X_{\Phi_{n_k}}) + g( \Phi_{n_k}(t, X_{\Phi_{n_k}}))\Big] \d t,
\end{eqnarray}
and
\begin{eqnarray}\label{2312-25eq4.53}
 && \E h(   X_{\Phi^*}(T))  \leq  \E  \liminf_{k\to\infty} h(  X_{\Phi_{n_k}}(T)) \leq \liminf_{k\to\infty} \E h(  X_{\Phi_{n_k}}(T)).
\end{eqnarray}
Hence, it follows from \eqref{2312-25eq4.52} and \eqref{2312-25eq4.53} that
$$J(\Phi^*) = \inf_{\Phi\in \mathcal{U}} J(\Phi),$$
which completes the proof of this theorem.
\end{proof}

\begin{remark}
In our optimal feedback  control problem $(\mathcal{P})$, the cost functional $J$ to be minimized contains the state and control variable in a separated form. However, we can consider a more general form for the cost functional as follows:
\begin{eqnarray}\label{2410-25eq-6.1}
  && J'(\Phi):= \E\bigg[ \int_0^T \Big[ f(s, X_\Phi(s), \Phi(s, X_\Phi(s))) \Big] \d s  + h(X_\Phi(T))  \bigg],
\end{eqnarray}
At this point, we replace Hypothesis \ref{2312-22hypo14} with the following hypothesis
\begin{hypothesis}\label{2410-25hypo-J2}
\begin{itemize}
\item[(F'.1)]   $f: [0,T]\times \V \times \H \to \R_+$ and $h:\V\to \R_+$.
\item[(F'.2)] $u_n\rightharpoonup u$ weakly in $\H$,   $X_n\to X$ strongly in $\V$ imply $f(s,X, u) \leq \lim\limits_{n\to\infty}\inf f(s,X_n, u_n)$, and the integral in \eqref{2410-25eq-6.1} exists and is finite.
\item[(F'.3)]  The function $h$ is sequentially lower semi-continuous.
	\end{itemize}
\end{hypothesis}
Actually, Hypothesis \ref{2410-25hypo-J2} remains the same as  Hypothesis \ref{2312-22hypo14}. 
To get the existence of our optimal feedback controls corresponding to the cost functional $J'(\Phi)$, we have to upgrade the proof in the Theorem \ref{2312-22th4.5}. 
In fact, we only need to modify \eqref{2312-25eq4.52} in the proof of the Theorem \ref{2312-22th4.5} as follows:
Taking into account the lower semicontinuity assumptions in our new Hypothesis \ref{2410-25hypo-J2} and the Fatou's lemma, we deduce
\begin{eqnarray}\label{2410-25eq6.52}
 \E \int_0^T \Big[f(t, X_{\Phi^*}, \Phi^*(t, X_{\Phi^*})) \Big]\d t  
&\leq& \E \int_0^T \liminf_{k\to\infty} \Big[f(t, X_{\Phi_{n_k}}, \Phi_{n_k}(t, X_{\Phi_{n_k}}))\Big] \d t  \nonumber\\
&\leq& \liminf_{k\to\infty}  \E \int_0^T \Big[f(t,   X_{\Phi_{n_k}}, \Phi_{n_k}(t, X_{\Phi_{n_k}}))\Big] \d t.
\end{eqnarray}
Therefore, there exists an optimal feedback control $\Phi^*\in \mathcal{U}$ for the new control problem corresponding to the cost functional $J'(\Phi)$.
\end{remark}

\section{Applications}

The results established in this paper are applicable to many interesting stochastic nonlinear evolution models under our hypotheses such as Hypothesis \ref{2312-22hypo1}. 
 In this section, we will provide a few examples to illustrate our results. In this section, we will list some existing control problem models in previous works that can fall within the framework we have set up, and in Section \ref{2410-25subsection5.2}, we will present an example which can not be covered in the framework previously in the existing literature.  For more example, one can see the book Liu and R\"{o}ckner  (2015).

It should be remarked that all the examples considered in Coayla-Teran (2020)  can be covered by our framework, including the controlled linear stochastic evolution model, the controlled stochastic reaction-diffusion model, the controlled stochastic nonlocal parabolic model, the controlled stochastic semilinear model. 
Lisei (2002) investigated the controlled tow dimensional stochastic Navier-Stokes equation, and demonstrated the existence of optimal controls to the controlled stochastic Navier-Stokes model. In fact, the coefficients of stochastic Navier-Stokes equations satisfy the fully local monotonicity conditions. Our results can cover the existence of optimal controls for stochastic Navier-Stokes equations in Lisei (2002).

\subsection{Controlled quasilinear stochastic partial differential equation}\label{2410-25subsection5.2}

Let $\Omega$ be a bounded domain in $\mathbb{R}^d$ with smooth boundary $\partial\Omega$. 
Set $H := L^2(\Omega)$ and $V := W^{1,\alpha}_0(\Omega)$, the usual Sobolev space with zero trace. Obviously,  we have the Gelfand triple $V \subseteq H \subseteq V^*$ By the Sobolev embedding theorem, and the embedding $V \subseteq H$ is compact. For $u, v \in V$,  we consider the following operator $A$
\begin{equation}\label{2410-25eq5.01}
\langle A(t, u), v \rangle = - \int_{\Omega} \Big[ \sum_{i=1}^{d} a_i(t, x, u(x), \nabla u(x)) \frac{\partial u(x)}{\partial x_i} + a_0(t, x, u(x), \nabla u(x)) v(x) \Big] dx. 
\end{equation}
Where, we assume that $a_i : [0, T] \times \Omega \times \mathbb{R} \times \mathbb{R}^d \rightarrow \mathbb{R}^d$, $i = 0, 1, 2, \ldots, d$, satisfy the following conditions: there exists a constant $\alpha$ satisfies
\begin{equation*}
  \left\{
     \begin{array}{ll}
       \alpha > 1, & \hbox{if $d = 1, 2$,} \\
       \alpha \geq \frac{2d}{d+2}, & \hbox{if $d \geq 3$,}
     \end{array}
   \right.
\end{equation*}
such that, for some nonnegative constants $c_j$, $j=1,2,3,4$ and $c>0$,
\begin{description}
\item[(H1)]   $a_i(t, x, u, z)$ is continuous in $(u, z) \in \mathbb{R} \times \mathbb{R}^d$ for a.e. fixed $(t, x) \in [0, T] \times \Omega$; $a_i(t, x, u, z)$ is measurable with respect to $(t, x) \in [0, T] \times \Omega$ for each fixed $(u, z) \in \mathbb{R} \times \mathbb{R}^d$.

\item[(H2)] There exists a function $f_1 \in L^{\frac{\alpha}{\alpha-1}}([0, T] \times \Omega, \mathbb{R}^+)$ such that for  all $(u, z) \in \mathbb{R} \times \mathbb{R}^d$, $i = 1, \ldots, d$,
\begin{equation}\label{2410-25eq5.02}
|a_i(t, x, u, z)| \leq c_1 |z|^{\alpha-1} + c_2 |u|^{\frac{\alpha-1}{d+2}} + f_1(t, x),  ~~ \text{a.e. for all $(t, x) \in [0, T] \times \Omega$}.  
\end{equation}

\item[(H3)] There exists a function $f_2 \in L^1([0, T] \times \Omega, \mathbb{R}^+)$ such that for  $(u, z) \in \mathbb{R} \times \mathbb{R}^d$,
\begin{equation}\label{2410-25eq5.03}
\sum_{i=1}^{d} a_i(t, x, u, z)z_i + a_0(t, x, u, z)u \geq c_3|z|^\alpha - c_4|u|^2 - f_2(t, x),  ~~ \text{a.e. for all $(t, x) \in [0, T] \times \Omega$}.    
\end{equation}

\item[(H4)]  Let $\gamma \in \big[0, \alpha \left(1 + \frac{2}{d}\right) - 2\big]$ and $f_3 \in L^1([0, T], \mathbb{R}^+)$. There exists a constant $c > 0$ such that for all $u, \tilde{u} \in \mathbb{R}$ and $z, \tilde{z} \in \mathbb{R}^d$,
\begin{eqnarray}\label{2410-25eq5.12}
&&\sum_{i=1}^{d} [a_i(t, x, u, z) - a_i(t, x, \tilde{u}, \tilde{z})](z_i - \tilde{z}_i) + [a_0(t, x, u, z) - a_0(t, x, \tilde{u}, \tilde{z})](u - \tilde{u}) \nonumber\\
&\geq& -c \left( f_3(t) + |u|^\gamma + |\tilde{u}|^\gamma \right) |u - \tilde{u}|^2, ~~~ \text{a.e. for all $(t, x) \in [0, T] \times \Omega$}.  
\end{eqnarray}

\end{description}

We consider the following controlled  quasi-linear stochastic partial differential equation:
\begin{eqnarray}\label{2410-25eq5.06}
\d u_\Phi(t, x) &=& \Big[\nabla \cdot a \left( t, x, u_\Phi(t, x), \nabla u_\Phi(t, x) \right) - a_0 \left( t, x, u_\Phi(t, x), \nabla u_\Phi(t, x) \right) + \Phi(t, u_\Phi(t,x))\Big] \d t  \nonumber\\
&&   + B(t, u_\Phi(t,x)) \d W(t)
\end{eqnarray}
with the zero Dirichlet boundary conditions, where $u : [0, T] \times \Omega \rightarrow \mathbb{R}$, the vectors $a = (a_1, a_2, \ldots, a_d)$ and $\nabla u(t, x) = \left( \frac{\partial_i u(t, x)}{\partial x_i} \right)_{i=1}^d$ is the gradient of $u$ with respect to the spatial variable $x$.

Recall the Gagliardo–Nirenberg inequality for $1 \leq p \leq \infty$,
\begin{equation}\label{2410-25eq5.07}
\|u\|_{L^p(\Omega)} \leq C \|\nabla u\|^\delta_{L^\alpha(\Omega)} \|u\|^{1-\delta}_{L^2(\Omega)}, 
\end{equation}
where $\delta \in [0, 1]$ and $\frac{1}{p} = \frac{\delta}{\alpha} - \frac{1}{d} \delta + \frac{1 - \delta}{2}.$
Then it follows from (H2) and \eqref{2410-25eq5.07} that $A$ is a measurable mapping from $[0, T] \times V$ to $V^*$, and moreover,
\begin{equation}\label{2410-25eq5.09}
\|A(t, u)\|_{V^*} \leq c_1 \|u\|^\alpha_V + c \|u\|^\alpha_V \|u\|^{\frac{2\alpha}{d}}_{H} + F(t),  
\end{equation}
where $F(t) = \int_{\Omega} f_1(t, x)^\frac{\alpha}{\alpha-1} dx$ is integrable on $[0,T]$.
And, it follows from (H4) and \eqref{2410-25eq5.07} that the operator $A$ satisfies
\[
\langle A(t, u) - A(t, v), u - v \rangle \leq \left( f_3(t) + C\|u\|^\gamma_{V} \|u\|^{1-\gamma}_{H} + C\|v\|^\gamma_{V} \|v\|^{1-\gamma}_{H} \right) \|u - v\|^2_{H},
\]
with $\delta = \frac{\alpha d}{\alpha d + 2\alpha - 2d}$. 

In this case, consider our optimal feedback control problem $(\mathcal{P})$, all conditions in Section 2 are satisfied. 
Thus, there is an optimal control $\Phi$ for our optimal feedback control problem $(\mathcal{P})$.


\vspace{0.5cm}

\noindent\textbf{Acknowledgements} This work is supported by Shandong Province Social Science Planning Research Project (No.24CJJJ18).

\addcontentsline{toc}{section}{Declarations}
\section*{Declarations}
\textbf{Conflict of interest}
The authors have no competing interests to declare that are relevant to the content of
this article.

\noindent\textbf{Data}
Data sharing not applicable to this article as no datasets were generated or analysed during the current study.

\end{document}